\documentclass[11pt]{amsart}

\usepackage{euscript}
\usepackage{amsmath}
\usepackage{amsthm}
\usepackage{epsfig}
\usepackage{amssymb}\usepackage{amscd}
\usepackage{epic}
\usepackage{eufrak}

\numberwithin{equation}{section}
\numberwithin{equation}{subsection}

\theoremstyle{plain}
\newtheorem{theorem}[equation]{Theorem}
\newtheorem{lemma}[equation]{Lemma}
\newtheorem{proposition}[equation]{Proposition}

\newtheorem{corollary}[equation]{Corollary}

\newtheorem{conjecture}[equation]{Conjecture}

\newtheorem{prop}[equation]{Proposition}

\theoremstyle{definition}
\newtheorem{example}[equation]{Example}
\newtheorem{remark}[equation]{Remark}

\newtheorem{definition}[equation]{Definition}

\newtheorem{rem}[equation]{Remark}
\newtheorem{bekezdes}[equation]{}

\numberwithin{equation}{section}
\numberwithin{equation}{subsection}

\DeclareMathOperator{\rank}{{\rm rank}}

\newcommand{\bH}{{\mathbb H}}

\def\C{\mathbb C}
\def\Q{\mathbb Q}
\def\R{\mathbb R}
\def\Z{\mathbb Z}

\def\N{\mathbb N}

\def\P{\mathbb P}

\newcommand{\calv}{{\mathcal V}}
\newcommand{\cala}{{\mathcal A}}

\newcommand{\calf}{{\mathcal F}}
\newcommand{\calO}{{\mathcal O}}

\newcommand{\frs}{{\frak s}}\newcommand{\frI}{{\frak I}}\newcommand{\frl}{{\frak l}}

\newcommand{\caln}{\mathcal{N}}

\newcommand{\X}{\tilde X}

\newcommand{\labelpar}{\label}
\newcommand{\m}{\mathfrak{m}}

\newcommand{\bfp}{p}

\newcommand{\frsw}{\mathfrak{sw}}

\newcommand{\wt}{{\rm wt}}

\title{The geometric genus of hypersurface singularities}

\author{Andr\'as N\'emethi}
\address{A. R\'enyi Institute of Mathematics, 1053 Budapest, Re\'altanoda u. 13-15, Hungary.}
\email{nemethi.andras@renyi.mta.hu}

\author{Baldur Sigur\dh sson}
\address{Central European University, Budapest and A. R\'enyi Institute of Mathematics, 1053 Budapest,
Re\'altanoda u. 13-15,  Hungary.}
\email{sigurdsson.baldur@renyi.mta.hu}
\thanks{The first author is partially supported by OTKA Grant 100796.
The second author is supported by the PhD program of the CEU, Budapest and
by the `Lend\"ulet' and ERC program `LTDBud' at R\'enyi Institute. }



\date{}

\begin{document}

\maketitle


\pagestyle{myheadings} \markboth{{\normalsize
A. N\'emethi and B. Sigur\dh sson}}{ {\normalsize Geometric genus of hypersurface singularities }}

\begin{abstract}
Using the path lattice cohomology we provide a conceptual topological characterization of the
geometric genus for certain complex normal surface singularities with rational homology sphere links,
which is uniformly valid for all superisolated and Newton non--degenerate hypersurface
singularities.
\end{abstract}

\section{Introduction}\labelpar{s:1}

\subsection{}
In this introduction we present the main result of the manuscript, for detailed definitions,
motivations, historical remarks and examples see the next section.

In the last years several conjectures and theorems target the topological characterization
of the geometric genus $p_g$
of complex normal surface singularities with rational homology sphere links.
They are usually formulated for certain families, and any
attempt  to find uniform characterization failed.

In order to have a chance for such a characterization, one has necessarily to assume two restrictions,
an analytical one and a topological one. The {\it Casson
Invariant Conjecture} (CIC) of Neumann and Wahl (\cite{NW}) predicts that for a complete intersection with
integral homology sphere link $p_g$ can be determined from the Casson invariant of the link  (see
\ref{con:CIC} here). This was generalized to rational homology sphere links by the first
author and Nicolaescu
\cite{[51]}; the {\it Seiberg--Witten Invariant Conjecture} (SWIC) connects $p_g$ with
the Seiberg--Witten invariant of the link (associated with the canonical $spin^c$--structure), see \ref{con:SWIC}. The predicted formula was
proved for several analytic families (e.g. rational, minimally elliptic, weighted homogeneous, splice quotient singularities), nevertheless it failed even
for some sporadic  hypersurfaces: namely, for certain superisolated  singularities.

The present note aims to find a uniform treatment for these counterexamples, and proposes
a new conceptual topological candidate for $p_g$, which is valid even for other important families of hypersurface singularities, e.g. for those with Newton non-degenerate principal part.
The main ingredient of the topological characterization is the {\it
path lattice cohomology} associated with the
link (or, with the negative definite lattice of a fixed resolution graph).

Recall that the lattice cohomology $\{\bH^q_{red}(M)\}_{q\geq 0}$ of the link (introduced in
\cite{Nlat}) is a  new categorification of the Seiberg--Witten invariant,
that is, its `normalized' Euler characteristic $eu(\bH^*(M))$ is  the Seiberg--Witten invariant.
In the $p_g$--comparisons the main dominating term is the first module $\bH^0_{red}(M)$, and in the superisolated case the non-vanishing of the next  terms are responsible for the failure of the SWIC.

Accordingly, the proposed new invariant  targets  a
 different version of the lattice cohomology, which concentrates
only on the $q=0$ part, and even optimizes it along different `paths'. A path is a sequence
of integral cycles supported on the exceptional curve of a fixed resolution, at each step
 increasing only by a base element, and connecting the trivial cycle with the anticanonical cycle.
For such a path $\gamma$ one defines a path lattice cohomology $\bH^0(\gamma)$, and one takes its
normalized rank $eu(\bH^0(\gamma))$. Then one shows that for any analytic type one has
$p_g\leq \min_\gamma eu(\bH^0(\gamma))$, hence it provides a natural topological upper bound for the
geometric genus. (The authors do not know if    $\min_\gamma eu(\bH^0(\gamma))$ can be defined by any other
construction, say, using gauge theory or low dimensional topology.)

The main result of the article is the following.

\begin{theorem}\labelpar{th:mainINTR} Assume that $(X,0)$ is a normal surface singularity whose link is
rational homology sphere.
Then the identity $p_g=\min_\gamma\, eu(\bH^0(\gamma))$ is true in the following cases:

(a) if $\bH^q(M)=0$ for $q\geq 1$ and the singular germ satisfies the SWIC Conjecture
(in particular, for all weighted homogeneous and minimally elliptic singularities);

(b) superisolated singularities (with arbitrary number of cusps);

(c) singularities with non-degenerate Newton principal part.

\noindent
Moreover, since the conjecture is stable with respect to equisingular deformation of hypersurfaces,
the conjecture remains valid for such deformations of any of the above cases.
\end{theorem}

\vspace{2mm}

The next section contains all the necessary definitions, main guiding examples, and status quo of the
problem. Section 3 contains the proof for superisolated germs. In this case the link is a surgery
3--manifold. The proof has  two non--trivial ingredients, already present in the recent literature:
the first one provides the lattice cohomology of surgery 3--manifolds \cite{NR}, the other is an
application of $d$--invariant vanishing result of certain $L$--space surgery 3--manifolds in
Heegaard Floer knot theory \cite{BL}.
The next sections contain the proof of the Newton non-degenerate case:
it involves deeply the very specific combinatorics of the associated toric resolution, and a lattice point
counting.

\section{Geometric genus formulae, conjectures, guiding examples}\labelpar{s:2}
\subsection{Preliminaries: the geometric genus.}\labelpar{ss:pre}
Let us fix a complex analytic normal surface singularity $(X,0)$.
Let $M$ be its link, the oriented smooth  3--manifold which is the boundary of a convenient
small representative $X$ of the germ. Since the real cone over $M$  is homeomorphic to $X$,
$M$ characterizes completely the local topology of $(X,0)$.
If we consider a resolution of $X$ with dual resolution graph $G$,
then $M$ can be realized as a plumbed 3--manifold associated with  $G$, and, in fact,
it contains the same information as $G$ itself (cf. \cite{NP}).
The {\it topological invariants } of the germ $(X,0)$ are read either from the
topology of $M$, or from the combinatorics of $G$.

The {\it analytic invariants} of $(X,0)$ are a priori associated with the
analytic structure of $(X,0)$ read e.g.  from the local algebra $\calO_{X,0}$, or from
the analytic sheaves of a resolution $\widetilde{X}\to X$ of $X$.
The very first one, and  probably the most important one, is the geometric genus
 $p_g:= \dim H^1(\widetilde{X},\calO_{\widetilde{X}})$. It guides (partially)
 the classification of singular germs and their deformation theory
\cite{Artin66,Laufer72,Laufer77,LaW,LaS,Ninv,Yau3,Yau1,WahlAnnals},
it is the local analog of the global Todd index of complex manifolds.
 As a `local index', it has several key connections with other numerical invariants as well
(see e.g. (\ref{eq:durfee})).  Usually, the geometric genus cannot be determined from the link,
even if we consider rather `simple' singularities. For example, the hypersurface singularities
$\{x^2+y^3+z^{18}=0\}$ and $\{z^2=y(x^4+y^6)\}$ have the same link but their $p_g$ are
3 and 2 respectively (cf. \cite[4.6]{NINV}).

Nevertheless, there is a strong belief, seriously supported by the results of the last decade,
that under some restrictions, $p_g$ can be determined from $M$. First of all, one needs to assume that $M$
is a {\it rational homology sphere}, or, equivalently, the resolution exceptional divisor is a tree
 of rational curves (in the above example the first Betti number $b_1(M)$ of $M$ is 2).
This is still not enough. E.g., consider an elliptic singularity with $p_g\geq 2$
(and even with integral homology sphere link), say for example the
hypersurface $x^2+y^3+z^{13}=0$ with $p_g=2$. Then, by \cite[4.1]{Laufer77},
the generic (non--Gorenstein)
analytic structure supported by the same topological type has $p_g=1$
(see also \cite{Ninv}). For other pairs with the same integral homology sphere link, but with different geometric genus see \cite{LMN}.  Hence,
one needs to add some analytic restriction too to guarantee the topological characterization of $p_g$.

In the choice of the analytic structures one
 possibility would be to consider a  {\it generic} one
on each irreducible component of the moduli space of analytic structures --- which, by the semicontinuity
of $p_g$ \cite{Elkik}, would provide the smallest $p_g$ of that moduli component. But,
it is equally  challenging (and this is  our interest here)
 to search for the geometric genus of special families of germs, which are the candidates providing
 the topological upper bound for $p_g$.
They can be related either  with special properties of $(X,0)$
(e.g. hypersurface, ICIS, Gorenstein, $\Q$--Gorenstein),
or with special constructions (see \ref{ex:examples} below).

\begin{example}\labelpar{ex:examples}
The next families will play  a key role in the next discussions.

(a) \  {\it Splice quotient} singularities
were introduced by Neumann and Wahl \cite{NW,NWnew,NWnew2}, their graph $G$ needs to satisfy  some
arithmetical  properties, which allow one to write down from the  combinatorics of $G$
the equations of the universal abelian cover of $(X,0)$ (up to equisingular deformation),
together with the corresponding action of $H_1(M,\Z)$ on it.
They  generalize the weighted homogeneous germs, but the equations associated
with different nodes might have different weights and degrees.
By construction their geometric genus depends only on $G$, the precise expression
is given in \cite{NO2} as an answer to Conjecture \ref{con:SWIC}.

(b) \ {\it Superisolated} singularities were introduced by I. Luengo \cite{Ignacio},
and they played a crucial role
in several testing procedures or counterexamples
\cite{Artal,aclm,LMN,Melle}.
If $C$ is a projective reduced plane curve with homogeneous equation
$f_d$ of degree $d$, and $f_{d+1} $ is a generic homogeneous equation of degree $d+1$,
then $f=f_d+f_{d+1}:(\C^3,0)\to (\C,0)$ is called  superisolated. Its link and $p_g$
 are independent of the choice of $f_{d+1}$ (see e.g. \cite{LMN}); in fact,
 $p_g=d(d-1)(d-2)/6$. We will assume that $C$ is irreducible.
  Then the link of $f$  is $\Q HS^3$ if and only if
$C$ is a rational cuspidal curve. We refer to the number of cusps as $\nu$.

(c) \ For {\it Newton non-degenerate hypersurfaces} see e.g. \cite{Kou}. Their principal
part contains
monomials  situated on a fixed Newton diagram $\Gamma$ and have generic coefficients.
Their geometric genus can be recovered as the number of lattice points with all positive entries and
which are `not above $\Gamma$' \cite{MT}. For more see Section \ref{s:NN1}.
\end{example}

\bekezdes{}
If $(X,0)$ is a hypersurface singularity in $(\C^3,0)$ then the
topological and analytic  invariants are strongly related with
those provided by the {\it embedded topological type} of $(X,0)$, that is, with the topology
of the embedding $M\subset S^5$, and the numerical invariants of the Milnor fibration.
Recall that the second Betti number of the Milnor fiber $F$ is the Milnor number $\mu$,
the intersection form on $H_2(F)$ determines the Sylvester invariants  $\mu_+,\ \mu_- $ and $\mu_0$,
while the signature
is defined by $\sigma=\mu_+-\mu_-$. Modulo the link $M$, the numerical invariants $p_g$, $\mu$ and $\sigma$
are related by two identities. Indeed, if $K$ is the canonical class/cycle on
$\widetilde{X}$ and $|\calv|$ is the number of vertices of $G$, then $K^2+|\calv|$
is a well--defined topological invariant of $M$,
and one has the following identities (valid in fact for any smoothing of a Gorenstein $(X,0)$)
\cite{Du,Laufer77b,St.Proc,W,LW}:
\begin{equation}\label{eq:durfee}
 \mu=12p_g+K^2+|\calv|-b_1(M),\ \ \ \ \ \ \ \ 
 -\sigma= 8p_g+K^2+|\calv|.
\end{equation}
Hence, if any of $p_g,\ \mu$ or $\sigma$ can be described from $M$ than the same is true
for all of them.   Note that $\mu$ can be recovered from the {\it embedded topological type}
(as the second Betti number of the universal cover of $S^5\setminus M$), hence this fact remains true for
$p_g$ as well. 

\bekezdes{}\label{bek:NOTG} {\bf Notations regarding $G$.} Regarding the link we will need the following
notations and terminology.
We fix a  resolution $\pi:\widetilde{X}\to X$  with resolution graph $G$ as above.
{\it We assume that $M$ is a $\Q HS^3$.}
Consider the lattice $L=H_2(\widetilde{X},\Z)$; it is freely generated by
$\{E_v\}_{v\in {\mathcal V}}$, the
irreducible components of the exceptional divisor $E:=\pi^{-1}(0)$ of $\pi$.
It is known that $G$ is connected and $L$ (that is, the intersection form
$\frI:=\{(E_v,E_u)\}_{v,u}$) is negative definite.
The determinant of the graph $G$ is defined as the absolute value of $\det(\frI)$.

Since we treat mainly hypersurface singularities, which are Gorenstein, we will assume that
 $G$ is {\it numerically Gorenstein}. This means that the {\it canonical cycle}
$K$, which satisfies the  system of
{\it adjunction relations} $(K+E_v,E_v)=- 2$ for all $v$,
is an integral cycle of $L$.

We define  $\chi:L\to \Z$ by $\chi(l)=-(l,l+K)/2$. (This is the Riemann--Roch formula:
for $l$ effective $\chi(l)$ is the analytic Euler characteristic of $\calO_l$.)
Set $\m:=\min_{l\in L}\chi(l)$.

If  $l'_k=\sum_vl'_{kv}E_v$ for $k=1,2$, then we write
$\min\{l'_1,l'_2\}:= \sum_v\min\{l'_{1v},l'_{2v}\}E_v$, and  $l'_1\leq l'_2$ if
$l'_{1v}\leq l'_{2v}$ for all $v\in {\mathcal V}$.
The valency of the vertex $v\in\calv$ is denoted by $\delta_v$.

We will write $Z_K:=-K$. Furthermore, we will assume that
$G$ is {\it a minimal good (resolution) graph}.  In such a case
one has the following.

\begin{lemma}\labelpar{cor:zke}
 Either  $(X,0)$ is  rational (equivalently A-D-E, hence $Z_K=0$)
or $Z_K > E$. Moreover, in the second case, the support of $Z_K-E$ is connected.
\end{lemma}
\begin{proof} It is known that the only rational numerically Gorenstein graphs  are of type
 A-D-E. Otherwise, under the assumption that the graph is a tree of rational vertices,
 $Z_K\geq E$ by \cite[2.8]{PP}.
 But $Z_K=E$ cannot happen. Indeed, $p_g=h^1(\calO_{Z_K})=
h^1(\calO_E)=0$ would imply rationality,  hence $Z_K=0$.
For the connectivity, see e.g. \cite[2.10]{Veys} or \cite[2.6]{PP}.
\end{proof}

\bekezdes{}
In the last years there was an intense activity to identify $p_g$ with certain
ingredients of the Seiberg--Witten
(or other equivalent/similar) theories. This worked nicely for several analytic structures, but failed for
some others. Though the goal of the present note is to present the parallel theory for those cases
which fail the `Seiberg--Witten connection', for a complete picture
we need to review certain notions from this part as well.

\subsection{Preliminaries: Seiberg--Witten invariant and lattice cohomology of $M$}
\labelpar{ss:SWLC}
Here is a short review of the lattice cohomology and path lattice cohomology
 of $M$. For more details  see  \cite{OSZINV,[69],Nlat,Nexseq}.

\bekezdes\labelpar{bek:LC} {\bf The lattice cohomology.}
$\Z^s\otimes \R$ has a natural decomposition into cubes. The
0--dimensional cubes are  the lattice points
$\Z^s$. Any $l\in \Z^s$ and subset $I\subseteq {\mathcal J}$ of
cardinality $q$  define a $q$--dimensional cube,  which has its
vertices in the lattice points $(l+\sum_{j\in I'}E_j)_{I'}$, where
$I'$ runs over all subsets of $I$. We define the
weight of  any such cube $\square_q$ by
\begin{equation*}
w(\square_q):=\max\{\chi(v) \,:\, \mbox{$v$ is a vertex of \,$\square_q$}\}.
\end{equation*}
The lattice cohomology (associated with the canonical $spin^c$ structure of $M$) is defined as follows.
For each $N\in\Z$, define
$S_N\subset \R^s$ as the union of all the cubes $\square_q$
(of any dimension) with $w(\square_q)\leq N$. Clearly, $S_N=\emptyset$,
whenever $N<\m$. Then for any $q\geq0$, set
\begin{equation*}
\bH^q(G):=\oplus_{N\geq \m}H^q(S_N,\Z), \ \ \
\bH^q_{red}(G):=\oplus_{N\geq \m}\widetilde{H}^q(S_N,\Z).
\end{equation*}
Then $\bH^q$ is  $2\Z$--graded, the $d=2N$--homogeneous elements
consist of $H^q(S_N,\Z)$. Also, $\bH^q$ is a $\Z[U]$--module:
the $U$--action is given by the restriction map
 $H^q(S_{N+1},\Z)\to H^q(S_N,\Z)$.
 Moreover, for $q=0$, a
 base--point $l\in S_{\m}$ provides an augmentation $H^0(S_N,\Z)=\Z\oplus\widetilde{H}^0(S_N,\Z)$,
 hence an augmentation of the graded $\Z[U]$--modules
$\bH^0=(\oplus_{N\geq \m}\Z)\oplus\bH^0_{red}$.
The graded $\Z[U]$ modules $\bH^*$ and  $\bH^*_{red}$  are called the lattice cohomology and the
reduced lattice cohomology of $G$.
They  depend only on $M$, and $\bH^*_{red}$ is a finite $\Z$--module.

\bekezdes\label{bek:SW} {\bf The Seiberg--Witten invariant.}
Recall that the Seiberg--Witten invariants of the oriented 3--manifold $M$
are rational numbers $\frsw_{\frs}(M)$ associated with the $Spin^c$--structures $\frs$ of $M$.
They can be recovered as (normalized) Euler characteristics of different cohomology theories,
e.g. for their relation with the Heegaard Floer homology see  \cite{OSz,OSz7}).
In the sequel we consider only the canonical  $spin^c$ structure, hence the symbol $\frs$ will be omitted.

By \cite{NSW}  the normalized Euler characteristics of the lattice cohomology
also agrees with the Seiberg--Witten invariant:
\begin{equation}\label{th:SW}
-\frsw(M)-(K^2+|\calv|)/8= eu(\bH^*(M)),\end{equation}
where $eu(\bH^*(M)):=-\m+
\textstyle{\sum_q}(-1)^q\rank_\Z \bH^q_{red}(M)$. Later it will be convenient to use the following
notation as well:
$eu(\bH^0(M)):=-\m+\rank_\Z \bH^0_{red}(M)$.

\subsection{The path lattice cohomology} \labelpar{bek:PCoh}
The search for a topological upper bound for $p_g$ lead to the definition of the path
(lattice) cohomology (in fact, this was the starting point of the lattice cohomology as well) \cite{OSZINV,Nlat}.

Consider a sequence $\gamma :=\{l_i\}_{i=0}^t$, $l_i\in L$ such that $l_0=0$, $l_t=Z_K$, and $l_{i+1}=l_i+
E_{v(i)}$ for some vertex $v(i)\in \calv(G)$. This defines a {\it path} (or 1--dimensional simplicial complex)
with  0--cubes  $\{l_i\}_i$ and  1-cubes  $[l_i,l_{i+1}]$.
We can repeat the construction of the lattice cohomology, but now only for those cubes which are
supported by $\gamma$. Indeed, let $\m_\gamma=\min_i\chi(l_i)$, and
 set $S_N^\gamma$ as the union of cubes supported by
$\gamma$ and with weight $\leq N$. Then one defines
$$\bH^q(\gamma)=\oplus _{N\geq \m_\gamma} H^q(S_N^\gamma, \Z),
\ \ \
\bH^q_{red}(\gamma)=\oplus _{N\geq \m_\gamma} \widetilde{H}^q(S_N^\gamma, \Z).$$
It turns out that $\bH^q(\gamma)=0$ for $q\not=0$, $\bH^0(\gamma)=(\oplus _{N\geq \m_\gamma}\Z)\oplus
\bH^0_{red}(\gamma)$, and $ \bH^0_{red}(\gamma)$ is a finite $\Z$--module.
Similarly, as for the lattice cohomology, we set
$$eu(\bH^0(\gamma)):=-\m_\gamma+\rank_\Z \bH^0_{red}(\gamma).$$
One verifies (see \cite[3.5.2]{Nlat}) that
\begin{equation}\label{eq:LCw}
eu(\bH^0(\gamma))=\sum_{i=0}^{t-1} \max\{0,\, \chi(l_{i})-\chi(l_{i+1})\}.
\end{equation}
There is a natural cohomological morphism $r^*:\bH^0(G)\to \bH^0(\gamma)$  induced by the restriction.
Usually it is neither  injective nor surjective. Nevertheless, $r^*$ is onto
for certain well--chosen  paths. Moreover, if $r^*$ is onto, then by \cite[3.5.4]{Nlat}
$eu(\bH^0(\gamma))\leq eu(\bH^0(G))$, hence
\begin{equation}\label{eq:ineq}
\min_{\gamma} \ eu(\bH^0(\gamma))\leq eu (\bH^0(G)).
\end{equation}
 Intuitively, $eu(\bH^0(G))$ depends on those lattice points $\{l_m\}_{m\in M}$
 of $L$ which realize  the `local   minima' of  $\chi$, and also on the `optimal' connecting paths of
 these points: for each pair $l_m$ and $l_{m'}$ there is a minimal $N(m,m')$ such that
 $l_m$ and $l_{m'}$ can be connected by a path in $S_{N(m,m')}$.
 On the other hand, $eu(\bH^0(\gamma))$ keeps from these data only those ones which are supported on
 $\gamma$, and $\min_\gamma \,eu(\bH^0(\gamma))$ minimizes  the sum
 $\sum_{i=0}^{t-1}  \max\{0,\, \chi(l_{i})-\chi(l_{i+1})\}$ among all the possible paths $\gamma$.

\bekezdes\labelpar{bek:PCAn}{\bf The analytic interpretation of $\min_\gamma \,eu(\bH^0(\gamma))$.}
For any analytic realization, by Riemenschneider-Kodaira vanishing
$h^1(\widetilde{X}, \calO_{\widetilde{X}}(-Z_K))=0$,
hence  $p_g=h^1(\calO_{Z_K})$.

Next, consider a sequence
  $\{l_i\}_{i=0}^t$, $l_i\in L$ such that $l_0=0$, $l_t=Z_K$, and $l_{i+1}=l_i+
E_{v(i)}$ as above. Recall $E_j\simeq \P^1$ for all $j$, hence $\chi(l_{i+1})-\chi(l_i)=1-(E_{v(i)},l_i)$.
Then, for any $0\leq i<t$ the
exact sequence $0\to \calO_{E_{v(i)}}(-l_i)\to \calO_{l_{i+1}}\to \calO_{l_i}\to 0$  induces
\begin{equation}\label{eq:PG}
h^1(\calO_{l_{i+1}})-h^1(\calO_{l_i})\leq h^1(\calO_{E_{v(i)}}(-l_i))= \max\{0,\, \chi(l_{i})-\chi(l_{i+1})\}.
\end{equation}
Taking the sum one obtains  $h^1(\calO_{Z_K})\leq  eu(\bH^0(\gamma))$ for any  path $\gamma$, hence
\begin{equation}\label{eq:PG2}
p_g\leq  \min_\gamma eu(\bH^0(\gamma)).
\end{equation}
Equality holds if for some $\gamma$  the cohomology  exact sequences split for all $i$.

Usually the concrete computation of $\min_\gamma\, eu(\bH^0(\gamma))$ is rather difficult.

\subsection{Some conjectures, results and examples.}\labelpar{ss:conj}
We review in short some key steps
in the topological characterization of the geometric genus via the Seiberg--Witten invariant.
We start with the {\it Casson Invariant Conjecture} (CIC) of Neumann--Wahl:
\begin{conjecture}\labelpar{con:CIC} \cite{NW} Consider an isolated complete intersection singularity
with signature $\sigma$, and whose link is an integral homology sphere with Casson invariant
$\lambda(M)$. Then  $\sigma/8=\lambda(M)$.
\end{conjecture}
The conjecture was verified
for Brieskorn--Hamm complete intersections and for those  hypersurfaces which are suspensions
of irreducible plane curve singularities \cite{NW}, see also \cite{FS}.
 Note that via (\ref{eq:durfee})(b), the  identity can be replaced
by $p_g=-\lambda(M)-(K^2+|\calv|)/8$, a version independent of any smoothing.
This version was  verified for splice quotient singularities (without the ICIS assumption) in  \cite{NO1}.
The original CIC  is still open.

One of the difficulties of a possible proof is the lack of
any characterization/description of {\it integral homology sphere}
hypersurfaces or complete intersection links  other than  iterated cyclic covers (for the behavior of
$\lambda(M)$ for such covers, see e.g.
Collin--Saveliev \cite{CoS,CoS2}). Having no other examples in hand,
it is hard to decide whether the validity of the conjecture is guaranteed merely by the special properties of
cyclic covers, or it covers a much deeper geometrical phenomenon. Moreover,
integral homology sphere links appear rather rarely
 (e.g. among the Newton nondegenerate hypersurfaces all germs with  $\Z HS^3$
links are of  Brieskorn type, while among rational graphs there is only one, namely the $E_8$).
Hence, it was necessary to extend the above conjecture to {\it rational homology sphere links}.
The conjectured identity was proposed by N\'emethi--Nicolaescu:
\begin{conjecture}\labelpar{con:SWIC} \cite{[51]}
Assume that the link $M$ of a normal surface singularity is a rational homology sphere, whose
Seiberg--Witten invariant (associated with the canonical $spin^c$--structure) is
$\frsw(M)$. If its analytic structure is `nice' then $p_g=-\frsw(M)-(K^2+|\calv|)/8$.
\end{conjecture}
The conjecture is proved for splice quotient singularities (including all rational,
minimal elliptic and weighted homogeneous singularities) \cite{BN,NCL,NO2}, and for suspensions of
irreducible plane curve singularities \cite{[55]}. It was extended to the equivariant case
(targeting the Seiberg--Witten invariant of all $spin^c$--structures and equivariant geometric genus
of the universal abelian cover \cite{BN,NCL}).

On the other hand, there are even hypersurface singularities which do not satisfy the conjecture.
The typical counterexample are   the superisolated singularities with $\nu\geq 2$ \cite{LMN}.

\begin{example}\label{ex:SI}
Assume that $(X,0)$ is a superisolated singularity with $\Q HS^3$ link.

(a) If $\nu=1$, that is, $C$ is unicuspidal, whose local cusp has local
irreducible plane curve singularity knot $K\subset S^3$, then $M=S^3_{-d}(K)$.
In \cite{BLMN2} (see also \cite{BLMN1,BLMN2,NR}) it is proved that the statement
of Conjecture \ref{con:SWIC} is equivalent with a `Density Property' of the semigroup of the local cusp.
This last property was checked in \cite{BLMN1} for `all known' curves $C$ via case-by-case verification,
and it was also proved recently in \cite{BL} using
the $d$--invariant of Heegaard Floer theory. Hence $p_g=eu(\bH^*(M))$.

Here we wish to emphasize an important point. The link $M$ with $\nu=1$ is `almost rational', that is,
modifying the resolution graph at only on vertex we can get a rational graph(see \cite{[63]}). In consequence, see
\cite{Nexseq}, one has the vanishing $\bH^q(M)=0$ for $q\geq 1$.
Therefore, $eu(\bH^*(M))=eu(\bH^0(M))$.
In particular, in this case
\begin{equation}\label{eq:eu}
p_g\leq \min_\gamma eu(\bH^0(\gamma))\leq eu(\bH^0(M))=eu(\bH^*(M))=p_g,\end{equation}
hence everywhere we must have equality.

(b) Nevertheless, for $\nu\geq 2$ counterexamples for Conjecture \ref{con:SWIC} exist \cite{LMN}.
In this case the above
vanishing has the weaker form: $\bH^q(M)=0$ only for  $q\geq \nu$ \cite{NR,Nexseq}.
Hence,  as we will see,   for superisolated singularities the non-vanishing of
$\bH^q(M)$ ($1\leq q<\nu$) obstructs the validity of Conjecture \ref{con:SWIC}.

Let us consider the case $C_4$ of \cite{LMN} (see also \cite[7.3.3]{Nlat}).
 $C$ has degree $d=5$ and two cusps, both with one Puiseux pair:
$(3,4)$ and $(2,7)$. The graph  $G$ is

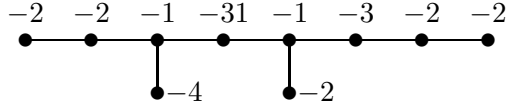
\begin{picture}(300,50)(20,-5)
\put(125,25){\circle*{5}} \put(150,25){\circle*{5}}
\put(175,25){\circle*{5}} \put(200,25){\circle*{5}}
\put(225,25){\circle*{5}} \put(150,5){\circle*{5}}
\put(200,5){\circle*{5}} \put(100,25){\line(1,0){175}}
\put(150,25){\line(0,-1){20}} \put(200,25){\line(0,-1){20}}
\put(125,35){\makebox(0,0){$-2$}}
\put(150,35){\makebox(0,0){$-1$}}
\put(175,35){\makebox(0,0){$-31$}}
\put(200,35){\makebox(0,0){$-1$}}
\put(225,35){\makebox(0,0){$-3$}} \put(160,5){\makebox(0,0){$-4$}}
\put(210,5){\makebox(0,0){$-2$}} \put(100,25){\circle*{5}}
\put(250,25){\circle*{5}} \put(275,25){\circle*{5}}
\put(100,35){\makebox(0,0){$-2$}}
\put(250,35){\makebox(0,0){$-2$}}
\put(275,35){\makebox(0,0){$-2$}}
\end{picture}

One shows that  $\m=-5$, $\rank_\Z(\bH^0)=5$, $\rank_\Z(\bH^1)=2$. Hence
$eu(\bH^0)=10$, but $eu(\bH^*)=8$.
\marginpar{verify}
Since for the superisolated germ with $d=5$ one has $p_g=10$, by equations (\ref{eq:ineq}) and (\ref{eq:PG2})
one gets  $\min_\gamma eu (\bH^0(\gamma))=10$ as well.
Hence, for this superisolated germ (\ref{eq:PG2}) is valid with equality, while Conjecture
\ref{con:SWIC} fails.

If we take any other analytic structure supported by the above graph,
by (\ref{eq:PG2})  $p_g\leq 10$ still holds, hence superisolated germs realize the optimal upper bound.

Note also that this topological type supports another natural analytic structure, namely a
splice quotient analytic  type: it is the
$\Z_5$--factor of the complete intersection
$\{z_1^3+z_2^4+z_3^5z_4=z_3^7+z_4^2+z_1^4z_2=0\}\subset (\C^4,0)$
by the diagonal action $(\alpha^2,\alpha^4,\alpha,\alpha)$
($\alpha^5=1$). By \cite{NO2} it satisfies the SWIC Conjecture \ref{con:SWIC},
hence $p_g=8$.

In particular, in their choices of the  topological characterization of their $p_g$,
some analytic structures
prefer $eu(\bH^*)$, some of them the extremal $\min_\gamma eu(\bH^0(\gamma))$
(and there might exists even other choices).

(c) We can ask whether the choice between $eu(\bH^*)$ and  $\min_\gamma eu(\bH^0(\gamma))$
is uniform in the case of hypersurface singularities, that is, if all hypersurfaces
choose $\min_\gamma eu(\bH^0(\gamma))$, as superisolated germs do. The answer is negative, even if
we assume additionally that the link is an integral homology sphere. For example,
take the suspension $f(x,y,z):= (x^3+y^2)^2+yx^5+z^{19}$ of an irreducible plane
curve singularity with 2 Puiseux pairs. The graph of $\{f=0\}$ is

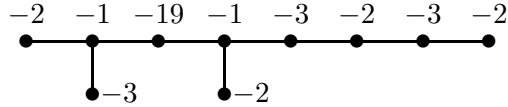
\begin{picture}(300,50)(20,-5)
\put(125,25){\circle*{5}} \put(150,25){\circle*{5}}
\put(175,25){\circle*{5}} \put(200,25){\circle*{5}}
\put(225,25){\circle*{5}} \put(150,5){\circle*{5}}
\put(200,5){\circle*{5}} \put(125,25){\line(1,0){175}}
\put(150,25){\line(0,-1){20}} \put(200,25){\line(0,-1){20}}
\put(125,35){\makebox(0,0){$-2$}}
\put(150,35){\makebox(0,0){$-1$}}
\put(175,35){\makebox(0,0){$-19$}}
\put(200,35){\makebox(0,0){$-1$}}
\put(225,35){\makebox(0,0){$-3$}} \put(160,5){\makebox(0,0){$-3$}}
\put(210,5){\makebox(0,0){$-2$}} \put(300,25){\circle*{5}}
\put(250,25){\circle*{5}} \put(275,25){\circle*{5}}
\put(300,35){\makebox(0,0){$-2$}}
\put(250,35){\makebox(0,0){$-2$}}
\put(275,35){\makebox(0,0){$-3$}}
\end{picture}

The following facts were checked by  Helge M\o ller Pedersen (via a  computer program based on the theoretical
facts of \cite{LN}):
$\m=-18$, ${\rm rank} \, \bH^0_{red}=26$, hence $eu ( \bH^0)=44$, and ${\rm rank}\, \bH^1=8$.
 Since the graph has only two nodes, one has $\bH^q=0$ for $q\geq 2$ (cf. \cite{Nexseq}),
 hence $eu(\bH^*)=44-8=36$. Moreover, $\min_\gamma eu(\bH^0(\gamma))=44$ too.

 On the other hand, the Milnor number of the plane curve singularity is 16, hence
 the Milnor number of $f$ is $\mu=16\cdot 18=288$. Then, by \ref{eq:durfee},
 one gets that $p_g=36$ (as expected, since this germ satisfies both conjectures
 \ref{con:CIC} and \ref{con:SWIC}, cf. \cite{[55]}).

 Hence, in this case $p_g=eu(\bH^*)=\min_\gamma eu(\bH^0(\gamma))-8$.
 Therefore, even for hypersurface singularities the two values  $eu(\bH^*)$ and
  $\min_\gamma eu(\bH^0(\gamma))$ might be different. 
\end{example}

\subsection{The new proposed identity.}\labelpar{ss:NEW}
 Having in mind the conclusion of Example \ref{ex:SI}(b)-(c), we can ask how accidental  the
superisolated example is. Or, what are the choices of other important families of hypersurfaces, e.g.
of the Newton non-degenerate ones. Here we wish to recall that
in \cite{BN} it is shown that for such germs from $M$ one can recover  the Newton diagram of the equation,
hence the equisingularity type of the germ too.
Hence, in principle, $p_g$ can be recovered from $M$; however this statement
 does not indicate any  topological candidate for $p_g$. (As a comparison,
  a similar statement regarding the possibility to recover
   the equisingularity type of suspensions of irreducible curves from their link
   is provided in \cite{[56]}. In that case the choice is $p_g=eu(\bH^*)$.)

The next theorem says that the extremal choice
$\min_\gamma eu(\bH^0(\gamma))$ is not accidental at all: in fact, all superisolated and
Newton non-degenearte germs prefer uniformly exactly this one.

\begin{theorem}\labelpar{th:main} Assume that $(X,0)$ is a normal surface singularity with
rational homology sphere link.
Then the identity $p_g=\min_\gamma\, eu(\bH^0(\gamma))$ is true in the following cases:

(a) if $p_g=eu(\bH^0(M))$. [This happens e.g.
whenever $\bH^q(M)=0$ for $q\geq 1$ and $(X,0)$ satisfies the SWIC (Conjecture \ref{con:SWIC}).
In particular, the conjecture is true for all weighted homogeneous and minimally elliptic singularities.]

(b) for superisolated singularities with arbitrary number of cusps
(in this case, in fact, $p_g=\min_\gamma\, eu(\bH^0(\gamma))=eu(\bH^0(M))$ too);

(c) for singularities with non-degenerate Newton principal part.

\noindent
Since the conjecture is stable with respect to equisingular deformation of hypersurfaces,
the conjecture remains valid for such deformations of any of the above cases.
\end{theorem}

\begin{corollary}
A superisolated singularity, where $C$ is an irreducible rational unicuspidal curve,
satisfies the Seiberg--Witten Invariant Conjecture \ref{con:SWIC} (predicted in
\cite{BLMN1,BLMN2} too).
\end{corollary}
Indeed, in this case the graph is `almost rational' (cf. \cite{[63]}),
hence $\bH^q(M)=0$ for $q\geq 1$ by \cite{Nexseq}. In particular,
$eu(\bH^*(M))=eu(\bH^0(M))$. Hence part (b) of the theorem suffices.
\bekezdes\labelpar{bek:wish}
 Let us repeat the meaning of the identity in Theorem~\ref{th:main}.
 For any sequence  $\gamma :=\{l_i\}_{i=0}^t$, $l_i\in L$ with
  $l_0=0$, $l_t=Z_K$, and $l_{i+1}=l_i+
E_{v(i)}$ ($v(i)\in \calv$) we set
$eu(\bH^0(\gamma)):=\sum_{i=0}^{t-1} \max\{0,\, \chi(l_{i})-\chi(l_{i+1})\}=
\sum_{i=0}^{t-1}\max\{0,\, -1+(l_i,E_{v(i)})\}$.
The statement is that $p_g=eu(\bH^0(\gamma))$ for a well--chosen path $\gamma$.

\bekezdes
Part (a) follows easily via the  inequalities (\ref{eq:ineq}) and (\ref{eq:PG2}).
By this part (a) we wish to emphasize that for several `simple cases',
one has all the equalities $p_g=\min_\gamma\, eu(\bH^0(\gamma))=eu(\bH^0(M))=eu(\bH^*(M))$.
(This explains why  in the earlier stage, when we had complete information only
about these simple  cases, it was difficult to predict the general behavior.)

Regarding part (a), note also that the condition  $p_g=\min_\gamma\, eu(\bH^0(\gamma))$ is more restrictive:
examples with  $p_g=\min_\gamma\, eu(\bH^0(\gamma))$ but with $p_g<eu(\bH^0(M))$ exist (see next example).

 \begin{example}\labelpar{ex:pgbH} \cite{LN} Consider the germ $(X,0)=\{x^{13}+y^{13}+x^2y^2+z^3=0\}$
 with non degenerate Newton principal part with $p_g=5$.
The graph is

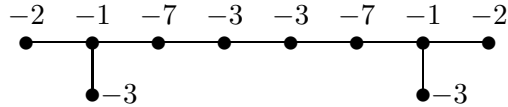
\begin{picture}(300,50)(20,-5)
\put(125,25){\circle*{5}} \put(150,25){\circle*{5}}
\put(175,25){\circle*{5}} \put(200,25){\circle*{5}}
\put(225,25){\circle*{5}} \put(150,5){\circle*{5}}
\put(275,5){\circle*{5}} \put(125,25){\line(1,0){175}}
\put(150,25){\line(0,-1){20}} \put(275,25){\line(0,-1){20}}
\put(125,35){\makebox(0,0){$-2$}}
\put(150,35){\makebox(0,0){$-1$}}
\put(175,35){\makebox(0,0){$-7$}}
\put(200,35){\makebox(0,0){$-3$}}
\put(225,35){\makebox(0,0){$-3$}} \put(160,5){\makebox(0,0){$-3$}}
\put(285,5){\makebox(0,0){$-3$}} \put(300,25){\circle*{5}}
\put(250,25){\circle*{5}} \put(275,25){\circle*{5}}
\put(300,35){\makebox(0,0){$-2$}}
\put(250,35){\makebox(0,0){$-7$}}
\put(275,35){\makebox(0,0){$-1$}}
\end{picture}

\noindent Then $\m=-1$, $eu(\bH^0)=6$, ${\rm rank}\, \bH^1=1$, 
and  $\min_\gamma eu(\bH^0(\gamma))=5$.
 \end{example}

\bekezdes The proof of part (b) of Theorem~\ref{th:main}
 will be given is section \ref{s:SI}. In fact,
we will show that for superisolated singularities $p_g=eu(\bH^0(M))$, hence (a) applies.

The Newton non-degenerate case is treated in the remaining sections.

\section{The proof of Theorem \ref{th:main} (b) for superisolated singularities.}\label{s:SI}

\subsection{} For the proof of Theorem \ref{th:main}(b) we have to combine
two results from the literature. In order to provide a more complete picture,
we give some additional details as well.

In the last years we witnessed  a focused effort to prove the (SWI) Conjecture
\ref{con:SWIC} for superisolated singularities with $\nu=1$ cusp,
or to understand the main cause of its
 failure in particular cases with $\nu\geq 2$. This materialized in several
 results \cite{BLMN1,BLMN2,[68],LMN,[63],[69],NR}
  and a new conjecture regarding the algebraic realization of cuspidal rational curves of
 fixed degree and given local singularities, the `semigroup distribution property' \cite{BLMN1,BLMN2}.
 These results and their reformulations in different languages
 of Heegaard Floer or Seiberg Witten theory
 (although they were originally motivated and guided by the SWIC and the distribution property),
reorganized and put together in a new puzzle provide the proof of Theorem \ref{th:main}.
This also shows that, in fact, for {\it certain  analytic structures} the statement of \ref{th:main}
is the {\it right guiding statement} and not the one in the  SWIC.

Let $f=f_d+f_{d+1}:(\C^3,0)\to (\C,0)$ be a superisolated singularity as in Example~\ref{ex:examples}(b),
where $C=\{f_d=0\}\subset \C\P^2$ is an irreducible rational cuspidal curve with $\nu$ cusps.
Let $(C,p_i)\subset (\C\P^2,p_i)$ be the local singularities of $C$, we denote their Milnor numbers by
$\mu_i$, and their local links by $K_i\subset S^3$.  Then the link of $(X,0)=\{f=0\}$
is $S^3_{-d}(K)$, where $K$ is the connected sum $K_1\#\cdots \# K_\nu$ . For the  plumbing
graph constructed  from the embedded resolution graphs of $(C,p_i)$
and the integer $d$,  see \cite[2.3]{NR}. The degree $d$ and the local
singularity types $(C,p_i)$ are related:
the rationality of $C$ implies $\sum_i\mu_i=(d-1)(d-2)$.
In particular, any statement like the SWIC
(or our main theorem) makes a bridge between the local topological types $(C,p_i)$ and the
global degree $d$. Computations show that for  general
$S^3_{-d}(K)$ (when the pair $(K,d)$ has no algebraic realization as above)
such identities cannot be expected. Hence such identities provide criterions for the algebraic realizations of
a degree $d$ curve with given local singularities.

\bekezdes The proof splits into two parts. The first part, done in \cite{NR},
 is valid for {\it any} surgery 3--manifold $S^3_{-d}(K)$.
It provides $\bH^*(S^3_{-d}(K))$, where $K$ is a connected sum of algebraic knots in $S^3$,
and $d$ is an arbitrary positive integer (and we do not assume the existence of a degree $d$ curve with
local types $K_i\subset S^3$, not even the identity $\sum_i\mu_i=(d-1)(d-2)$).
The cohomology is described completely in
terms of the semigroups $\{S_i\}_{i=1}^\nu$
of the algebraic knots $K_i\subset S^3$. The construction uses the `reduction theorem' of \cite{LN},
which describes the lattice cohomology (defined a priori in the lattice $L$)
in the first quadrant of a  lattice of rank $\nu$.
 The formula of $eu(\bH^0)$ is the following.
   For any $n\in \Z_{\geq 0}$ set
   $$\Sigma_n:=\{(\beta_1,\ldots,\beta_\nu)\in (\Z_{\geq 0})^\nu\, :\, \textstyle{\sum_i}
   \beta_i=n+1\},$$
   and for any $(\beta_1,\ldots,\beta_\nu)\in( \Z_{\geq 0})^\nu$ set the weight
   $$\overline{W}(\beta_1,\ldots,\beta_\nu):=
   \textstyle{\sum_{i=1}^\nu}\ |\ \{k_i\not\in S_i\,:\, k_i\geq \beta_i\}\ |.
   $$
   Then, by \cite[6.1.15]{NR} (see also Remark 6.1.19 in [loc. cit.])
\begin{equation}
\label{eq:NR}
eu(\bH^0(S^3_{-d}(K)))=\sum_{j\geq 0}\min\, \{\overline{W} \ \mbox{restricted to $\Sigma_{jd}$}\}.
\end{equation}
The sum in (\ref{eq:NR}) is finite: since $\{k_i\not\in S_i\,:\, k_i\geq \mu_i\}=\emptyset$,
if $jd+1\geq \sum_i\mu_i$ we get $\min(\{\overline{W}|\Sigma_{jd}\}=0$.
E.g., if $\sum_i\mu_i=(d-1)(d-2)$, then for $j\geq d-2$ we get zero contribution.
\begin{remark}
Though $\bH^0(\gamma)$ is not mentioned in \cite{NR},
in its  section 6 one can see clearly that
$\min_\gamma eu (\bH^0(\gamma))=eu (\bH^0)$ even in this general topological context; nevertheless, we do not
need this, in our case it will follow automatically from the estimate of the second part.\end{remark}

\bekezdes The second ingredient is the main result of
\cite{BL} (as the proof of the conjectured distribution property of \cite{BLMN1}), and it
is valid only in the presence of the algebraic  realization of $S^3_{-d}(K)$.
Its proof uses surgery properties of the
$d$--invariant of the Heegaard Floer theory.

 \cite[Theorem 5.4]{BL} reads as follows.   If for any $(\beta_1,\ldots,\beta_\nu)\in \Z^\nu$ one writes
   $$\overline{W}^*(\beta_1,\ldots,\beta_\nu):=
   \textstyle{\sum_{i=1}^\nu}\ |\ \{k_i\in \Z\setminus S_i\,:\, k_i\geq \beta_i\}\ |,
   $$
then
\begin{equation}
\label{eq:BL}
\min\, \{\overline{W}^* \ \mbox{restricted to $\Sigma_{jd}$}\}=(j-d+1)(j-d+2)/2.
\end{equation}
Since $S_i\subset \Z_{\geq 0}$, the minimum in (\ref{eq:BL}) is realized for
$(\beta_1,\ldots,\beta_\nu)\in (\Z_{\geq 0})^\nu$, hence the minimums in (\ref{eq:BL}) and
(\ref{eq:NR}) agree. (Indeed, if $\beta_1<0$ and $\beta_\nu>0$, then
$\overline{W}^*(\beta_1,\ldots,\beta_\nu)\geq \overline{W}^*(\beta_1+1,\ldots,\beta_\nu-1)$.)
Hence (\ref{eq:BL}) and (\ref{eq:NR}) combined give
$$eu(\bH^0(S^3_{-d}(K)))=\sum_{j=0}^{d-2}j(j+1)/2=d(d-1)(d-2)/6=p_g.$$
This with inequalities  (\ref{eq:ineq}) and (\ref{eq:PG2}) end the proof, cf. part (a).

\vspace{2mm}

For us, in fact, the main geometric meaning of the `arithmetical result' of \cite{BL} is exactly
the statement of Theorem~\ref{th:main}(b).

\section{Preliminaries on Newton non-degenerate singularities}\label{s:NN1}

\subsection{The Newton boundary \cite{Kou}}\labelpar{ss:NN}
For any  set  $S\subset\N^3$ denote by $\Gamma_+(S)
\subset \R^3$ the convex closure of $\bigcup_{\bfp\in
S}(\bfp+\R_+^3)$. The collection of all  boundary faces of
$\Gamma_+(S)$  is denoted by $\calf$, while the set of
\emph{compact} faces of \(\Gamma_+(S)\)  by $\calf_c$.
  By definition, the \emph{Newton boundary} (or
\emph{diagram}) \(\Gamma(S)\) associated with $S$ is the union of
compact boundary faces of $\Gamma_+(S)$.
Let $\Gamma_-(S)$ denote the cone with base \(\Gamma(S)\) and vertex \(0\).

Let $f\colon (\C^3,0) \to (\C,0)$ be an analytic function germ
defined by a convergent power series $\sum_\bfp a_\bfp z^p$,
where $\bfp=(p_1,p_2,p_3)$ and $z^\bfp:=z_1^{p_1}z_2^{p_2}z_3^{p_3}$. By definition, the
\emph{Newton boundary} $\Gamma(f)$ of $f$ is $\Gamma({\rm supp}(f))$,
where ${\rm supp}(f)$ is the support $\{\bfp : a_\bfp\neq0\}$ of $f$,
and we write $\Gamma_\pm(f)$ for $\Gamma_\pm({\rm supp}(f))$. The
\emph{Newton principal part} of $f$ is $\sum_{\bfp\in \Gamma(f)}
a_\bfp z^\bfp$. Similarly, for any $q$-dimensional face $\bigtriangleup $ of
$\Gamma(f)$, set
$f_\bigtriangleup(z):=\sum_{\bfp\in\bigtriangleup}a_\bfp
z^\bfp$. We say that $f$ is \emph{non-degenerate} on
$\bigtriangleup$ if the system of equations $\partial
f_\bigtriangleup/\partial z_1=\partial f_\bigtriangleup/\partial
z_2 =\partial f_\bigtriangleup/\partial z_3=0$
 has no solution in
$(\C^*)^3$. When $f$ is non-degenerate on every $q$-face
 of $\Gamma(f)$, we say (after Kouchnirenko
\cite{Kou}) that $f$ has a \emph{non-degenerate Newton
principal part}. In the sequel we assume that $f$ has this property. Moreover, we also
assume that $f$ defines an isolated singularity at the origin.
This can be characterized by $\Gamma(f)$ as follows: cf. \cite[1.13(ii)]{Kou} or \cite[2.1]{BN}:
 \begin{equation}\label{eq:IS}
 \begin{split}
\mbox{ $\Gamma(f)$ has a vertex on every coordinate plane, and it has }\\
\mbox{a vertex at most 1 far from any chosen coordinate axis.}\end{split}\end{equation}
Nevertheless, we will not assume that $\Gamma(f)$ is `convenient'
 (recall that $\Gamma(f)$ is  convenient if it intersects all the coordinate axes).
 In fact, if $f$ is not convenient, there are several ways to complete the
 diagram to a convenient one by not modifying the equisingularity type. Hence, in general,
 several diagrams might produce the same
 equsingularity type. From all  possible diagrams {\it we choose a minimal one} (with respect to the inclusion,
 for its existence and `almost unicity' see \cite[\S 3]{BN}). This minimal diagram, in general, is not
 convenient.

Furthermore, as always in this note, we assume that the link $M$ of $(X,0)=\{f=0\}$
is a \emph{rational homology
sphere}. This, in terms of  Newton boundary $\Gamma(f)$ reads as follows, cf.~\cite{Saito}:
\begin{equation}\label{eq:rhs}
\text{$M$ is a rational homology sphere}  \iff
 \Gamma(f)\cap \N_{>0}^3=\emptyset.
\end{equation}

\begin{lemma} \cite[\S 2.3]{BN}
  \label{lem:13} Under the above assumptions,
  if a face of $\Gamma(f)$ is not a triangle then it is a trapezoid.
  By permuting coordinates, its vertices are:
  \((p,0,n)\), \((0,q,n)\),
  \((r_1,r_2+tq,0)\) and
  \((r_1+tp,r_2,0)\), where \(p,q > 0\), \({\rm gcd}(p,q)=1\), \(t \geq 1\) and
  \(r_1,r_2 \geq 0\).

  For any face of the diagram at most one edge might have inner
   lattice points, and if an edge has inner lattice points then that edge sites
    in a coordinate plane.
\end{lemma}
In the sequel we denote by
$\langle p,q\rangle=\sum_ip_iq_i$ the standard scalar product on $\R^3$.
Let $N_\bigtriangleup$ be the primitive normal vector
of a 2--face $\bigtriangleup\in\calf$ oriented such that $\langle N_\bigtriangleup, (1,1,1)\rangle >0$.
In the case of two adjacent faces
$\bigtriangleup$ and $\bigtriangledown$ from $\calf$, we write $t_{\bigtriangleup,\bigtriangledown}-1$
 for the number of  interior lattice points of  the common edge. If both faces are compact then
 $t_{\bigtriangleup,\bigtriangledown}=1$ by (\ref{eq:rhs}). Furthermore,
 let $n_{\bigtriangleup,\bigtriangledown}$ be the greatest common divisor of the 2-minors of the
 vectors $N_\bigtriangleup$ and $N_\bigtriangledown$.
\subsection{Oka's algorithm for $G$}\labelpar{OKAsect}
Let $f\colon (\C^3,0) \to (\C,0)$ be a germ as in \ref{ss:NN}.
We recall the combinatorial algorithm of M. Oka  (based on  toric resolution),
which provides a dual resolution graph
$G$ of $(X,0)$ from $\Gamma(f)$, cf.  \cite[Theorem~6.1]{Oka}.

\bekezdes {\bf (The algorithm)}\labelpar{alg}
The graph $G$ is a subgraph of a larger graph
$\widetilde{G}$, whose construction is the following. We
start with a set of vertices, each of them corresponding to a face from
$\calf$ (we will call them
\emph{face vertices}).
Consider two adjacent faces
$\bigtriangleup$ and $\bigtriangledown$ from $\calf$.
Then we connect the corresponding vertices  by
$t_{\bigtriangleup,\bigtriangledown}$ copies of the
following chain.

If \(n_{\bigtriangleup,\bigtriangledown}>1\) then
let \(0 < c_{\bigtriangleup,\bigtriangledown} <
n_{\bigtriangleup,\bigtriangledown}\) be the
unique integer for which
\begin{equation}\label{eq:333}
N_{\bigtriangleup,\bigtriangledown} :=
(N_{\bigtriangledown} +
c_{\bigtriangleup,\bigtriangledown}
N_{\bigtriangleup})/n_{\bigtriangleup,\bigtriangledown}
\end{equation}
is an integral vector.  Let us write
\(n_{\bigtriangleup,\bigtriangledown}/
c_{\bigtriangleup,\bigtriangledown}\) as a (negative)
continued fraction:
\begin{equation}
  \label{eq:3}
  \frac{n_{\bigtriangleup,\bigtriangledown}}
  {c_{\bigtriangleup,\bigtriangledown}} =
   e_1 -
  \cfrac{1}{e_2 -\cfrac{1}{ \dotsb - \cfrac{1}{e_k}}}\ ,
\end{equation}
where each \(e_i \geq 2\).  Then the chain with the corresponding
self-intersection numbers is

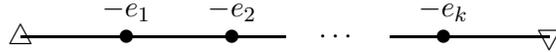
\begin{picture}(270,40)(0,5)
\put(150,20){\circle*{5}} \put(190,20){\circle*{5}}
\put(270,20){\circle*{5}} \put(110,20){\line(1,0){100}}
\put(250,20){\line(1,0){60}}
\put(110,20){\makebox(0,0){$\bigtriangleup$}}
\put(310,20){\makebox(0,0){$\bigtriangledown$}}
\put(150,30){\makebox(0,0){$-e_1$}}
\put(190,30){\makebox(0,0){$-e_2$}}
\put(270,30){\makebox(0,0){$-e_k$}}
\put(230,20){\makebox(0,0){$\cdots$}}
\end{picture}

The left ends of all the
$t_{\bigtriangleup,\bigtriangledown}$ copies of the
chain (marked by $\bigtriangleup$) are identified with the face
vertex corresponding to $\bigtriangleup$, and similarly for the right
ends marked by $\bigtriangledown$.

If \(n_{\bigtriangleup,\bigtriangledown}=1\) then
the chain consists of an  edge connecting the vertices
\(\bigtriangleup\) and \(\bigtriangledown\) (we put
$t_{\bigtriangleup,\bigtriangledown}$ of them). Also,
in this case we set
$c_{\bigtriangleup,\bigtriangledown} :=
0$ and $N_{\bigtriangleup,\bigtriangledown}
:= N_{\bigtriangledown}$.

Next, we compute the decoration
\(b_{\bigtriangleup}\) of any face vertex
\(\bigtriangleup\in\calf_c\) by the equation:
\begin{equation}
  \label{eq:4}
  b_{\bigtriangleup} N_{\bigtriangleup} +
  \sum_{\bigtriangledown\in\calf_\bigtriangleup}
  t_{\bigtriangleup,\bigtriangledown}
  N_{\bigtriangleup,\bigtriangledown} = 0,
\end{equation}
where $\calf_\bigtriangleup$ is the collection of all
2--faces of $\Gamma_+(f)$ adjacent to $\bigtriangleup$.

In this way we obtain the graph $\widetilde{G}$.
Notice that the face vertices corresponding to \emph{non-compact
faces} are not decorated.
\begin{proposition}\labelpar{prop:OKA}
(a) \cite{Oka} If we delete all the vertices corresponding to non--compact faces  (and all
the edges adjacent to them) we get a dual resolution graph $G$.

(b) \cite[4.2.5]{BN} Under the above choice of the `minimal' Newton diagram,
cf. \ref{ss:NN},
the graph $G$ is the minimal good resolution graph. In particular, the nodes (vertices with valency
$\delta_v\geq 3$) are exactly the face vertices associated with compact faces of the diagram.

(c) \cite[3.3.11]{BN} If $\bigtriangleup$ is a compact face and $\bigtriangledown$ is an adjacent
non-compact face, then \(n_{\bigtriangleup,\bigtriangledown}>1\). Hence, such an edge of  $\bigtriangleup$
produces a nontrivial chain (leg) of $G$.
\end{proposition}
We denote by $\{b_v\}_{v\in \calv}$ the decorations (self--intersections) of the corresponding vertices.
\bekezdes \labelpar{bek:embres}  {\bf The `extended' graph $G^e$.}
 Replace each edge of $\widetilde{G}$ connecting
a face vertex $\widetilde{v}$ of $\widetilde{G}\setminus G$
and another vertex  $w$ of $G$  by an arrow  with  arrowhead $a$ and
supporting vertex  $w$. Then $G^e$ consists of $G$ equipped with this type of arrowheads.

 The set of vertices of $G$ is denoted by
$\calv$, the set of arrowheads of $G^e$ by $\cala$, and we refer to $\calv^e:=\calv\cup \cala$ as the set of
(arrowhead and non-arrowhead) vertices of $G^e$. Let $\caln$ be the set of (non--arrowhead)
face vertices of $G^e$; these are the nodes in both graphs $G$ and  $G^e$.

To each $v\in \calv^e$ we associate a vector $N_v$ in $\Z_{\geq 0}^3$.
  If $v$ is a face vertex of $G$ corresponding to
$\bigtriangleup$ then $N_v=N_\bigtriangleup$.
 If $a\in \cala$ corresponds to $\bigtriangleup\in\calf\setminus \calf_c$
 then $N_a=N_{\bigtriangleup}$.
  All the other  vectors are determined in a unique way by
the next identities (see \cite{BN}):

Fix $v\in \calv$, and set $\calv^e_v:=\{w\in \calv^e: \, w \ \mbox{adjacent to $v$ in $G^e$}\}$
and $\calv_v:=\calv^e_v\cap \calv$. Then
\begin{equation}  \label{eq:5}
  b_v N_v + \sum_{w\in\calv^e_v}
  N_{w} = 0.
\end{equation}
(In this procedure it helps to know that
 the vector associated with  a neighbor of a face vertex $v$ corresponding to $\bigtriangleup$
on the chain in the direction  $w$ corresponding to $\bigtriangledown$
is  $N_{\bigtriangleup,\bigtriangledown}$.)

If $v$ is on a chain connecting the face vertices $w_1$ and $w_2$, then $N_v=r_1N_{w_1}+r_2N_{w_2}$ with
$r_1,r_2\in \Q_{>0}$. Since for any $v\in\caln$
all the entries of $N_v$  are strictly  positive, the same property remains
true for all $v\in \calv$.
By (\ref{eq:5}) (which typically characterizes the multiplicities of the function encoded in
an embedded resolution graph), or by \cite{Oka}, one has
\begin{equation}\label{eq:zi}
\mbox{the multiplicity of  $z_i$ along $E_v$  is the $i$-th coordinate of
$N_v$ ($v\in\calv$).}\end{equation}
In fact, if we decorate the (arrowhead and non--arrowhead) vertices of $G^e$ by the $i$-th
 coordinate of the vectors  $N_v$, then we get the embedded resolution graph
 of the germ $z_i$ on $(X,0)$, where the arrowheads with zero decorations can be deleted,
 arrowheads with positive decorations represent strict transforms of $z_i=0$
 which are not necessarily transversal with their supporting curves (the
 corresponding decoration is the
 intersection multiplicity of the strict transform with the supporting curve). This decoration
 can be larger than 1 if the corresponding non-compact face is not a coordinate plane.

\begin{definition}\labelpar{def:H} \
(i) For each $v\in \calv^e$ we
denote by $\ell_v$ the linear function $\langle N_v,\cdot\rangle$.

(ii) If  $l = \sum_{v\in \calv} l_v E_v\in L$, then we write  $m_v(l) = l_v$.
We extend this for any $v\in \calv^e\setminus \calv$ by  setting  $m_v(l):=-1$.
(For a motivation of the value $-1$ see the proof of Lemma \ref{lem:zkminuse}.)

(iii) For any $l\in L$ and $v\in \calv^e$
we define the half-space  $H^{\geq}_v(l):=\{p\in \R^3\,|\, \ell_v(p)\geq
m_v(l)\}$, and write  $H^=_v(l)$ for its boundary plane.
Moreover, we define
the {\it rational polytope}
$$\Gamma^e_+(l):=
\cap_{v\in\calv^e}\ H^{\geq}_v(l).$$
\end{definition}

\subsection{Divisorial valuations, weights and the canonical cycle}\labelpar{ss:weight}
Let $f$ be as in (\ref{ss:NN}) and
$L$ the  lattice of rank $|\calv|$ associated with $G$, with  $E_v$ as
its basis. The resolution with exceptional curve $E=\cup_vE_v$ is denoted by $\widetilde{X}$
(as in \ref{bek:NOTG}).
\begin{definition}\labelpar{def:wt}\
(a)  For a nonempty finite set $S$ in $\Z_{\geq 0}^3$ and  $v\in\calv$  define
 $\wt_v (S) = \min_{p\in S} \ell_v(p)$ and
  $\wt(S):=\sum_{v\in\calv}\wt_v(S)E_v\in L$.
 For  $0\neq h\in\C\{z\}$ set $\wt_v (h) = \wt_v({\rm supp}(h))$ and
$\wt(h) = \wt({\rm supp}(h))$.

 (b)  For any $\bar h \in\calO_{X,0}$,
  denote by ${\rm div}_v(\bar h)$ the order of vanishing  on $E_v$
  of the pullback of $\bar h$ to $\widetilde{X}$, and set also
  ${\rm div}(\bar h)=\sum_v{\rm div}_v(\bar h)E_v$.
  If $h\in\C\{z\}$, let ${\rm div}(h) = {\rm div}(h|_X)$.
\end{definition}

\begin{rem}\labelpar{rem:wtdiv}
The functions $\wt_v$ are order functions.
For any $p\in\N^3$, we have $\wt(z^p) = {\rm div}(z^p)$ (cf.
\ref{eq:zi}). Since for all $h_1,\, h_2\in \C\{z\}$ with ${\rm supp}(h_1)\cap {\rm supp}(h_2)=\emptyset$
one has $\wt(h_1+h_2)=\min \{\wt(h_1),\wt(h_2)\}$, and
the ${\rm div}_v$ are valuations, we get
$\wt(h) \leq {\rm div}(h)$ for any $h\in\C\{z\}$.
\end{rem}

\begin{lemma}\labelpar{prop:canonical} \cite[Theorem (9.1)]{Oka}
The (anti)canonical divisor $Z_K=-K\in L$ satisfies
\begin{equation*} \label{eq:canonical}
 Z_K-E = \wt(f) - \wt(z_1z_2z_3).
\end{equation*}
\end{lemma}

\begin{lemma}\labelpar{lem:zkminuse} $\Gamma^e_+(Z_K-E)=\Gamma_+(f)-(1,1,1)$.
\end{lemma}
\begin{proof} Let us analyse the non-compact faces of $\Gamma_+(f)$. They (up to a permutation of
coordinates) have the form $z_1+az_2=a$ for some $a\in \Z_{\geq 0}$.
This follows from (\ref{eq:IS}) or from \cite[3.1.2]{BN}.
This plane shifted by $(1,1,1)$
gives the plane $z_1+az_2=-1$, a fact which explaines
 Definition \ref{def:H} for $v\not\in\calv$ too.
Otherwise the statement follows from  Lemma \ref{prop:canonical}.
\end{proof}

\section{The proof of Theorem \ref{th:main}(c)}\labelpar{s:proof(c)}

Since part (a) includes the rational  hypersurfaces, we assume that $(X,0)$ is not rational. Let
$\widetilde{X}$ be a minimal good resolution (provided by  Prop. \ref{prop:OKA});
we write $\calO$ for $\calO_{\widetilde{X}}$. Let $G$ be the corresponding resolution graph.
By Lemma  \ref{cor:zke} we get $Z_K>E$.
\subsection{The sequence $\{z_i\}_i$ (preliminaries)}\labelpar{ss:zi}
\bekezdes\labelpar{bek:b1} From the definition of path cohomology,
by paragraphs \ref{bek:PCAn} and \ref{bek:wish},
we have to construct a sequence $\{l_i\}_{i=0}^t$,
$l_i\in L$, such that $l_0=0$, $l_t=Z_K$, $l_{i+1}=l_i+E_{v(i)}$ with
$$p_g=\sum_{i=0}^{t-1} h^1(\calO_{l_{i+1}})-   h^1(\calO_{l_{i}})=\sum_{i=0}^{t-1} h^1(\calO_{E_{v(i)}}(-l_i)).$$
\bekezdes\labelpar{bek:b2}
We need the  dual picture. Take  a sequence
 $\{z_i\}_{i=0}^t$,
$z_i\in L$, where $z_0=0$, $z_t=Z_K$, $z_{i+1}=z_i+E_{v(i)}$.
Then the pair of sheaves $\calO(-z_{i+1})\hookrightarrow   \calO(-z_{i})$ gives
 an exact sequence
\begin{equation}\label{eq:INPG}
0\to H^0(\widetilde{X}, \calO(-z_{i+1}))\to H^0(\widetilde{X}, \calO(-z_{i}))
\to H^0(\calO_{E_{v(i)}}(-z_i))\to \cdots \end{equation}
Since $H^0(\calO)/H^0(\calO(-Z_K))\simeq H^0(\calO_{Z_K})=\C^{p_g}$,
and $h^0(\calO_{E_{v(i)}}(-z_i))=\max\,\{0, 1-(E_{v(i)},z_i)\}$,
\begin{equation}\label{eq:PGC}
p_g=\sum_{i=0}^{t-1} \dim \, \frac{
H^0(\widetilde{X}, \calO(-z_{i}))}{ H^0(\widetilde{X}, \calO(-z_{i+1}))} \leq
\sum_{i=0}^{t-1}\max\,\{0, 1-(E_{v(i)},z_i)\} .\end{equation}
The dual statement of \ref{bek:b1}
 requires the existence of the sequence $\{z_i\}_i$ with equality in Equation (\ref{eq:PGC})
 for all $i$.
The duality is realized by $z_i:=Z_K-l_{t-i}$ and Serre duality.

[Indeed, if $z_{i+1}=z_i+E_{v(i)}$, $l_{j+1}=l_j+E_{u(j)}$, $z_i=Z_K-l_{t-i}$,
then with $j:=t-1-i$ one has $u(j)=v(i)$ and
$h^1(\calO_{E_{u(j)}}(-l_j))=h^0(\calO_{E_{v(i)}}(-z_i))$ by Serre duality.]

Since $H^1(\calO_{E_j})=H^1(\calO_E)=0$, in fact, in \ref{bek:b1} we need  a sequence $\{l_i\}_i$
which starts with $l_0=E$ and ends with $l_t=Z_K$ (otherwise is as in \ref{bek:b1}).

By dual considerations (since $H^0(\calO(-Z_K))\hookrightarrow H^0(\calO(-Z_K+E))$ is an isomorphism)
it is enough to construct a sequence $\{z_i\}_{i=0}^t$ which starts with $z_0=0$ and ends with
$z_t=Z_K-E$ (otherwise is as above).  This is what we will do.

\bekezdes\labelpar{bek:b3} {\bf The plan of the proof.} To a sequence $\{z_i\}_i$ of elements of $L$
(with $z_0=0$, $z_t=Z_K-E$, $z_{i+1}=z_i+E_{v(i)}$)  we associate the sequence
$\{\Gamma^e_+(z_i)\}_i$ of polytopes. Note that
$\Gamma^e_+(z_t)=\Gamma_+(f)-(1,1,1)\subset \Gamma^e_+(z_0)\subset (\R_{\geq -1})^3$.
Clearly $\Gamma^e_+(z_{i+1})\subset \Gamma^e_+(z_i)$.
This filtration  realizes a partition of
the lattice points
$(\Gamma_-(f)-(1,1,1))\cap \Z_{\geq 0}^3 $ by
\begin{equation}\label{eq:PI}
P_i:=(\, \Gamma^e_+(z_i)\setminus \Gamma^e_+(z_{i+1})\,)\cap \Z_{\geq 0}^3.
\end{equation}
Then, for the sequence $\{z_i\}_i$ provided by the algorithm \ref{def:alg} we show
that for $0\leq i<t$
\begin{equation}\label{eq:egy}
\max\,\{0, 1-(E_{v(i)},z_i)\}\leq |P_i|
\end{equation}
and
\begin{equation}\label{eq:ketto}
|P_i|\leq  \dim \, \frac{
H^0(\widetilde{X}, \calO(-z_{i}))}{ H^0(\widetilde{X}, \calO(-z_{i+1}))}.
\end{equation}
These two facts together with (\ref{eq:PGC}) show that in (\ref{eq:PGC}) we must have equality.

Note that based on (\ref{eq:egy}) and (\ref{eq:PGC}) together with the result of Merle--Teissier \cite{MT},
which says that $|\Gamma_-(f)\cap\Z_{>0}^3|=p_g$, or in the above terms
 $\sum_i|P_i|=p_g$,  we could already conclude our theorem. Nevertheless,
 providing an independent argument for the additional
(\ref{eq:ketto}),  besides the proof of our theorem we reprove the Merle--Teissier result as well.

\vspace{2mm}

Some of the steps can be analysed easily.
\begin{lemma}\labelpar{lem:easy}
Assume that along the sequence  one has $(E_{v(i)},z_i)>0$ for some $i$. Then $P_i=\emptyset$, hence
(\ref{eq:egy}) is valid. Moreover, in (\ref{eq:ketto}) one has equality (with both sides zero).
\end{lemma}
\begin{proof}
Assume that $P_i\not=\emptyset$,
set $p\in P_i$ and $v:=v(i)$. Then $\ell_v(p)=m_v(z_i)$ and $\ell_w(p)\geq m_w(z_i)$
for every $w\in \calv$. Moreover, since all the entries of $p$ are non-negative,
$\ell_w(p)\geq 0$ for every $w\in \calv^e\setminus \calv$ as well.
Therefore, by (\ref{eq:5}), $0=(b_v\ell_v+\sum_{w\in \calv^e_v}\ell_w)(p)
\geq b_vm_v(z_i)+\sum_{w\in \calv_v}m_w(z_i)=(E_v,z_i)>0$, a contradiction.
For the second statement use  (\ref{eq:INPG}).
\end{proof}
\subsection{The operation $l\mapsto c(l)$ and the ratio test.}\labelpar{ss:Laufer}
In the construction of the sequence $\{z_i\}_i$ a part of the `easy' steps are provided
by the next `completion operation'.

If $l=\sum_vm_vE_v\in L$, then we define the support of $l$ as $|l|:=\{v\in \calv\,:\, m_v\not=0\}$.

In the next proposition the meaning of the involved condition
is motivated  by the adjunction formula $(E_v,Z_K-E)=2-\delta_v$.
Note that $\delta_v\leq 2$ if $v\not\in\caln$.

\begin{proposition}\labelpar{prop:cl} \
(A) Fix $l\in L$.  Then there exists a unique
 element $c(l)\in L$ with
\begin{enumerate}
\item $m_n(c(l))=m_n(l)$ for all $n\in\caln$;
\item $(c(l),E_v)\leq 2-\delta_v$ for every $v\not\in\caln$;
\item $c(l)$ is minimal with properties (1) and (2).
\end{enumerate}
\noindent (Clearly, by (1), $c(l)$ depends only on $\{m_n(l)\}_{n\in\caln}$.)

\vspace{1mm}

\noindent (B) Let $l$ be
 as in (A), and $z\in L$ another cycle such that $z\leq c(l)$ and $m_n(z)=m_n(l)$ for all $n\in\caln$.
 We construct a `computation sequence' $\{x_j\}_{j=0}^T$ starting with $x_0=z$ as follows.
First, one takes $x_0=z$. Then, if $x_j$ is already constructed, and it does not satisfy
(A)(2), that is, there exists $v$ with $(x_j,E_v)>2-\delta_v$,
then $x_{j+1}=x_j+E_{v}$ for such  $v$. If $x_j$ satisfy (A)(2) then
we stop and $j=T$. Then, we claim that for any such computation sequence, $x_T=c(l)$.

\vspace{1mm}

\noindent (C) If $l_i\in L$,  such that $m_n(l_1)\leq m_n(l_2)$ for all $n\in\caln$, 
then $c(l_1)\leq c(l_2)$.

\vspace{1mm}

\noindent (D) If $G$ is not an $A_n$--graph, then $c(0)=0$. (For $A_n$ one has $c(0)=-E$).

\vspace{1mm}
\noindent (E)   $c(Z_K-E)= Z_K-E$.
\end{proposition}
\begin{proof}
The proof of (A)-(B) is an alteration of  Artin's proof of the existence of the fundamental cycle
\cite{Artin66} and of
the existence of the Laufer computation sequence providing this cycle \cite{Laufer72}.
The major steps are the following.
First, note that by the negative definiteness of the intersection form
there exists a cycle with properties (A)(1-2).
Then one shows (as in  \cite{Artin66})
that if  $c'(l)$ and $c''(l)$ satisfy
(A)(1-2) then $\min\{c'(l),c''(l)\}$ also satisfies them, hence there exists a unique minimal
element $c(l)$ satisfying (A)(1-2).
For (B), one checks
by induction that $x_j\leq c(l)$ for every $j$. Hence the sequence must stop and
 $x_T\leq c(l)$. Since
 $x_T$ satisfies (A)(1-2), and $c(l)$ is minimal with this property, $x_T=c(l)$.

(C)  $c(l_2)-(l_2-l_1)$ satisfies (A)(1-2) and restricted on $\caln$ is $l_1$, hence
$c(l_1)\leq c(l_2)-(l_2-l_1)$. 

(D) It is an elementary arithmetical verification on the chains of $G\setminus \caln$.

(E) Since  $Z_K-E$ satisfies the conditions (A)(1-2),  $l:=Z_K-E-c(Z_K-E)\geq 0$ by (A)(3).
But $(l,E_v)\geq 0$ for all $v$ it the support of $l$, hence by the negative definiteness of $G$
we have $l\leq 0$ too. Hence $l=0$.
\end{proof}


\begin{definition}\labelpar{ss:ratio}{\bf (The ratio test)} We fix  $l\in L$, and
we consider the ratio $R(v):=m_v(l)/m_v(Z_K-E)$ for every $v\in\calv$.
We say that the ratio test for $l$ chooses the vertex $v\in\calv$ if
$R(v)=\min\{R(w)\,|\,w\in\calv\}$ and $R(v)<R(w)<\infty$
for at least one  adjacent vertex $w$ of $v$.
(If $m_v(Z_K-E)=0$ then $R(v)=\infty$, which is `larger than any real number'.)
\end{definition}
\begin{lemma}\labelpar{lem:test}
Assume that for some effective  $l$ supported on $\caln$ one has $0<c(l)<Z_K-E$.
Then the ratio test for $c(l)$ always makes a choice, that is, it is not possible that
$v\mapsto R(v)$ is a constant on $|Z_K-E|$. Moreover, a chosen $v$ is in the support of
$Z_K-E-c(l)$.
\end{lemma}
\begin{proof}
Since the support of $Z_K-E$ is connected, cf. \ref{cor:zke},
if $R(v)$ is constant on this support then  $c(l)=r(Z_K-E)$ for some $r\in(0,1)$
(in the complement of this support both $c(l)$ and $Z_K-E$ are zero). Then for any $w$ with
$\delta_w=1$ we would get $r=r\cdot (Z_K-E,E_w)=(c(l),E_w)\in\Z$, a contradiction. Hence,
a choice $v$ exists. Moreover, $m_v(c(l))=m_v(Z_K-E)$ would imply
$1=R(v)={\rm min}_w\{R(w)\}$, or
$c(l)\geq Z_K-E$.
\end{proof}

\subsection{The algorithm of $\{z_i\}_i$}\labelpar{ss:princ}

\begin{definition}\labelpar{def:alg} The sequence $\{z_i\}_{i=0}^t$
is constructed by the following steps.

{\bf S1.} $\bar{z}_0=0$.

{\bf S2.} Assume that $\bar{z}_i$ is already constructed, and $\bar{z}_i<Z_K-E$.
If $i=0$ then we choose an
 arbitrary $E_v$ in the support of $Z_K-E$ and we write $E_{v(0)}:=E_v$. If $i>0$, then $\bar{z}_i>0$, and
the ratio test for $\bar{z}_i$ makes a choice (cf. Lemma \ref{lem:test}), say $v(i)$.
Then (in both cases) set $\bar{z}_{i+1}=c(\bar{z}_i+E_{v(i)})$.
Note that by Lemma \ref{lem:test} and Proposition
\ref{prop:cl} one has $ \bar{z}_i+E_{v(i)}\leq \bar{z}_{i+1}\leq Z_K-E$.

{\bf S3.} If $\bar{z}_{i+1}<Z_K-E$ then run Step
S2 again. If $\bar{z}_i=Z_K-E$ then stop and put $\bar{t}:=i$.
\end{definition}

Note that different choices might produce different sequences $\{\bar{z}_i\}_{i=1}^{\bar{t}}$.

For any $i\geq 0$ one has  $c(\bar{z}_i+E_{v(i)})\geq  \bar{z}_i+E_{v(i)}$. For $i=0$ this follows from
Proposition~\ref{prop:cl}(D), while for $i>0$ from part (C). Hence by \ref{prop:cl}(B)
  there exists  an `intermediate' sequence between
$\bar{z}_i+E_{v(i)}$ and $\bar{z}_{i+1}$.  Completing the above sequence with these
  intermediate steps we get the wished sequence $\{z_i\}_{i=0}^t$ for which
  $z_{i+1}-z_i$ is always a base element. Along these intermediate steps, cf.
  Lemma \ref{lem:easy}, the inequalities (\ref{eq:egy})--(\ref{eq:ketto}) are true with equalities,
  since $(E_{v(i)},x_j)>2-\delta_{v(i)}\geq 0$.

The very first step $0=z_0\mapsto z_1=E_{v(0)}$ is also easy: $(E_{v(0)},z_0)=0$,
$H^0(\calO)/H^0(\calO(-E_{v(0)}))$ is 1--dimensional, and since for any $v\in\calv$ all the entries
of $ N_v$ are positive, $(N_v,p)>0$ for all $\Z_{\geq 0}^3\setminus (0,0,0)$. Hence $P_0=\{(0,0,0)\}$, and it
has cardinality 1.

The next lemma analyzes the step $\bar{z}_i\mapsto \bar{z}_i+E_{v(i)}$ ($i>0$).
\begin{lemma}\labelpar{lem:stepbar} $\delta_{v(i)}=2$ cannot happen. If $\delta_{v(i)}=1$ then
$(\bar{z}_i,E_{v(i)})=1$.
\end{lemma}
\begin{proof} The choice of $v(i)$ implies
$$\frac{(\bar{z}_i, E_{v(i)})}{\sum m_w(\bar{z}_i)}
>\frac{(Z_K-E,E_{v(i)})}{\sum m_w(Z_K-E)}
= \frac{2-\delta_{v(i)}}{\sum m_w(Z_K-E)}\ ,$$
where the sums run over $w\in \calv_{v(i)}$. Then use the fact that $\bar{z}_i$ is $c$ of an element.
\end{proof}
\bekezdes\labelpar{bek:faces}
This shows (using again  Lemma \ref{lem:easy}) that what remains to verify is the validity of
(\ref{eq:egy})--(\ref{eq:ketto})
in the step $\bar{z}_i\mapsto \bar{z}_i+E_{v(i)}$,
whenever $i>0$, $v(i)$ is a face vertex  (that is, node with $\delta_{v(i)}\geq 3$)
and $(\bar{z}_i,E_{v(i)})\leq 0$. In the sequel we assume all these facts.

\vspace{2mm}

It is also convenient to define for any $n\in\caln$ and cycle $l>0$
\begin{equation*}\label{eq:FI}
F_n(l):=\big(   \cap _{w\in \calv\setminus n} H^{\geq }_w(l)\big)\cap H^=_{n}(l),\ \ \mbox{and } \ \
F_n^{nb}(l):=  \big( \cap _{w\in \calv_{n}}  H^{\geq }_w(l)\big)\cap H^=_{n}(l).
\end{equation*}
Let  $C_n$ (resp. $C_n^{nb}$) be the real cone over $F_n(Z_K-E)$ (resp.
 $F_n^{nb}(Z_K-E)$) with vertex 0.

The definition of $F_n(l)$ should be compared with the definition of the
 face $F^e_n(l)$ of $\Gamma^e_+(l)$ which sits in $H^=_n(l)$, defined by
 $\big(\cap _{w\in \calv^e\setminus n} H^{\geq }_w(l)\big)
 \cap H^=_{n}(l)$ (which, having extra equations indexed by $\calv^e\setminus \calv$,
   is a subset of $F_n(l)$). 
Also, $F_n(l)\subset F_n^{nb}(l)$ since $\calv_n\subset \calv\setminus n$.

\begin{lemma}\label{lem:ZKE} Fix $n\in \caln$. Then
$ F_n^{nb}(Z_K-E)=F_n(Z_K-E)=F^e_n(Z_K-E)$. 
In particular, $C_n^{nb}=C_n$.
\end{lemma}
\begin{proof} With simplified notations we have
$F^{nb}_n\supset F_n\supset F^e_n$. Next note
that $\calv^e_n=\calv_n$, cf. \ref{prop:OKA}(c), that is, if
 $n'\in\calv^e\setminus \calv$ is connected by a chain to $n$, then this chain is
  non--empty.

 If $n'$ is another face vertex of $\calv$ connected to $n$ by a chain
 with vertices $v_1,\cdots ,v_k$, then all the intersections $H^=_{n}(Z_K-E)\cap H^=_{v_j}(Z_K-E)$ agree with the
 intersection  $H^=_{n}(Z_K-E)\cap H^=_{n'}(Z_K-E)$.
 This follows from the adjunction relation and identities (\ref{eq:5}).
 The same is true if
 $n'\in\calv^e\setminus \calv$, and $ H^=_{n'}(Z_K-E)$ is $\{p: \ell_{n'}(p)=-1\}$ (compatibly with
 the convention $m_{n'}=-1$, cf. Definition~\ref{def:H}(ii)).

 Therefore, $F^{nb}_n$ equals that polygon of $H_n^=$ which is cut out by the intersections of type
 $H_{n'}^{\geq}$, where $n'$ are as in the previous paragraph (neighbor extended face--vertices).
 But, by  Lemma~\ref{lem:zkminuse},
 this equals  the corresponding face of $\Gamma_+(f)-(1,1,1)$,  and also  $F^e_n$, since
 the $n$--face of $\Gamma_+(f)$ is also cut out by the equations given by  neighbour face--vertices  $n'$.
\end{proof}
\bekezdes\label{again}
Let us consider again the step $\bar{z}_i\mapsto \bar{z}_i+E_{v(i)}$ of the algorithm with restrictions as in
Paragraph \ref{bek:faces}. 

We still denote
$P_i:=(\Gamma^e_+(\bar{z}_i)\setminus \Gamma^e_+(\bar{z}_i+E_{v(i)}))\cap \Z_{\geq 0}^3$.

In the affine plane $H_i:=H^=_{v(i)}(\bar{z}_i)$ there are several
polygons:

$\bullet$  $F_i:=F_{v(i)}(\bar{z}_i)$ cut out by the half--spaces indexed by $\calv\setminus v(i)$;

$\bullet$  $F_i^{nb}:=F^{nb}_{v(i)}(\bar{z}_i)$ cut out by the half--spaces indexed by the `neighbours' $\calv_{v(i)}$;

$\bullet$  $F_i^{cn}:=C_{v(i)}\cap H_i$ cut out by the cone over $F_{v(i)}(Z_K-E)$.

\vspace{2mm}

By Lemma~\ref{lem:ZKE}, the cone $C_{v(i)}$ is given by the inequalitites
\begin{equation*}\label{eq:con}
\frac{\ell_w(p)}{m_w(Z_K-E)}\geq \frac{\ell_{v(i)}(p)}{m_{v(i)}(Z_K-E)}
\ \ \ \mbox{for all $w\in\calv_{v(i)}$}.
\end{equation*}
This restricted to $H_i=\{l_{v(i)}(p)=m_{v(i)}(\bar{z}_i)\}$ transforms into
\begin{equation}\label{eq:cone}
\ell_w(p)\geq \frac{m_{v(i)}(\bar{z}_i)}{m_{v(i)}(Z_K-E)} \cdot \frac{m_w(Z_K-E)}{m_{w}(\bar{z}_i)}
\cdot m_w(\bar{z}_i)= r_{i,w}\cdot  m_w(\bar{z}_i) \ \ \ \ \ \ (w\in\calv_{v(i)}).
\end{equation}
By the ratio test $r_{i,w}\leq 1$. The inequalities which cut out $F_i^{nb}$ in $H_i$ are
$\ell_w(p)\geq m_w(\bar{z}_i)$ ($w\in\calv_{v(i)}$), hence they are more restrictive
than those from (\ref{eq:cone}). Hence  $F_i^{nb}\subset F_i^{cn}$.

For  technical purposes we need another polygon  $F_i^{cn-}$ satisfying
$F_i^{nb}\subset F_i^{cn-}\subset F_i^{cn}$.

First we associate to any $w\in\calv_{v(i)}$ an integer $\epsilon_j\in\{0,1\}$.
Let us fix $w$ and let  $I_w$ be that edge of $F_i^{cn}$ which is determined by
equality $\ell_w(p)=r_{i,w}\cdot  m_w(\bar{z}_i)$. Then $\epsilon_w=0$ except if the following happens:
$I_w$ contains at least one  lattice point (hence $r_{i,w}\cdot  m_w(\bar{z}_i)\in \Z$),
and $r_{i,w}\cdot  m_w(\bar{z}_i)< m_w(\bar{z}_i)$. In this case $\epsilon_w=1$.
We define $F_i^{cn-}\subset H_i$ cut out by the inequalities $\ell_w(p)\geq r_{i,w}\cdot  m_w(\bar{z}_i)+\epsilon_w$
for all $w\in\calv_{v(i)}$.  Hence
\begin{equation}\label{eq:FINCL}
F_i\subset F_i^{nb}\subset F_i^{cn-}\subset F_i^{cn}.\end{equation}
 The next lemma compares the lattice points of these polygons.

\begin{lemma}\labelpar{lem:cone}
$P_i=F_i\cap \Z_{\geq 0}^3 =F_i^{nb}\cap \Z_{\geq 0}^3=F_i^{cn-}\cap \Z_{\geq 0}^3$.
\end{lemma}
\begin{proof}
The first equality  follows from the fact that $\ell_{v(i)}$ takes integral values on the lattice points,
hence all the lattice points of
$\Gamma^e_+(\bar{z}_i)\setminus \Gamma^e_+(\bar{z}_i+E_{v(i)})$ are on the boundary face
 $F_i$.

 For the other  ones, via (\ref{eq:FINCL}), it is enough to show that
 $F_i^{cn-}\cap \Z_{\geq 0}^3\subset P_i$. Take $p_0\in F_i^{cn-}\cap \Z_{\geq 0}^3$
 and we wish to show that $p_0\in P_i$.

 Run the algorithm from $z_0=0$ to $z_t=Z_K-E$. Replacing $\Gamma^e_+(z_k)$
 by $\Gamma^e_+(z_{k+1})$, we say that we pass the lattice points in their difference.
 Along the
 algorithm we have to pass all the lattice points of $(\Gamma_+(f)-(1,1,1))\cap \Z_{\geq 0}^3$,
  each of them exactly once.
 By the above discussion (see also \ref{lem:easy}), we pass a lattice point only at step of type
 $\bar{z}_j\mapsto \bar{z}_j+E_{v(j)}$, where $v(j)\in\caln$. Along this step we pass exactly the lattice points
 $P_j$. Moreover, by (\ref{eq:FINCL})   applied  for $v(j)$,
 these lattice points are situated in the cone $C_{v(j)}$.

 Hence, if our chosen  lattice point
 $p_0$ is situated only in the cone $C_{v(i)}$, then we can pass it only at step $v(i)$ by the
 plane $H_i$, hence it is in $P_i$.

 Assume next that $p_0$ is situated at the intersection of two cones $C_{v(i)}$ and
 $C_{n'}$, hence on  an edge $I_w$ of $F_i^{cn}$ ($w$ being on the chain connecting $v(i)$ and $n'$).

 Then along the algorithm we pass $p_0$ either at the step $v(i)$ by $H_i$, or before this step.
 In the first case $p_0\in P_i$. We show that the second case cannot happen.

 Indeed, if we pass $p_0$ at step $j<i$ with $v(j)=n'$, then
 $\ell_{v(j)}(p_0)=m_{v(j)}(\bar{z}_j)$. Since $\bar{z}_i\geq \bar{z}_j+E_{v(j)}$, one gets that
  $\ell_{v(j)}(p_0)<m_{v(j)}(\bar{z}_i)$.

  \vspace{2mm}

  \noindent {\bf Claim.}  $\ell_{v(i)}(p_0)=m_{v(i)}(\bar{z}_i)$ and
   $\ell_{v(j)}(p_0)<m_{v(j)}(\bar{z}_i)$ implies    $\ell_{w}(p_0)<m_{w}(\bar{z}_i)$.

 \vspace{2mm}

 \noindent {\it Proof of Claim}:
Write $n=v(i)$. Claim follows from the existence of  positive rational numbers $a$ and $b$ with
($\dagger$) $\ell_{w}=a\ell_n+b\ell_{n'}$ and  ($\ddagger$) $m_{w}\geq am_n+bm_{n'}$.

The numbers $a$ and $b$ are the following.
Let $(w,v_1,\ldots, v_k)$ be the chain connecting $n$ and $n'$.
Let $d$ and $a'$ be the determinant of the chains
$(w,v_1,\ldots, v_k)$ and     $(v_{1},\ldots, v_k)$  (where
the determinant of the empty graph is 1). Then  $a=a'/d$ and $b=1/d$.
This follows from weighted summation of identities (\ref{eq:5}) for ($\dagger$) and of inequalities (\ref{prop:cl})(2)
 for  ($\ddagger$).

\vspace{2mm}

Hence, we have $\ell_{w}(p_0)<m_{w}(\bar{z}_i)$, which implies that $\epsilon_w=1$, hence
by the construction of $F_i^{cn-}$ we get $p_0\not\in F_i^{cn-}$, a contradiction.
\end{proof}
\begin{remark}
By Oka's algorithm \ref{alg}, the greatest common divisor of $2$--minors of
the $2\times 3$--matrix form by $N_{v}$ and $N_w$ of two neighbor vertices $(v,w)$ is one.
 Hence $\ell_w$ is primitive linear function on $H_i$, and $F_i^{cn-}\cap \Z^3_{\geq 0}$
 is obtained from $F_i^{cn}\cap\Z^3_{\geq 0}$ by eliminating
 $I_w\cap \Z^3_{\geq0}$ with $\epsilon_w=1$ (and no other points).
\end{remark}

\subsection{Polygons in an affine space}\labelpar{ss:POLY}\
In this section we establish some combinatorial facts about triangles
and trapezoids in affine planes of $\R^3$. (For motivation see Lemma \ref{lem:13}.)

\begin{definition} \label{def:standard_small}
A triangle or trapezoid $\bigtriangleup_{st}$
 in an affine plane $A_{st}$ is \emph{standard} if its vertices are integral points,
it contains no integral points in its interior and at most one of its edges
contain integral points in its interior.
If a standard polygon has an edge with integral points in its interior, call
that side the \emph{long} edge, the others the \emph{short} edges.
Otherwise, any edge may be referred to as long or short. In this way, all
standard polygons $\bigtriangleup_{st}$ have a long edge.
Let  $t(\bigtriangleup_{st})-1$ be the number of interior lattice points of the long edge.
(In the case of a trapezoid the long edge is automatically parallel to another edge, cf. Lemma \ref{lem:13}.)

Let $\bigtriangleup_{st}$ be  a standard polygon  with
edges $\{I_w\}_w$. To each edge $I_w$  we associate an integral affine  function
$\frl_{st,w}$, which is constant along $I_w$, $\frl_{st,w}|_{\bigtriangleup_{st}}\geq \frl_{st,w}|_{I_w}$,
 and its restriction to $A_{st}$ is primitive. ($\frl_{st,w}$ is primitive if its restriction
$\frl_{st,w}:A_{st}\cap \Z^3\to\Z$ is onto.)
The restriction of $\frl_{st,w}$ to $A_{st}$  is unique up to an additive integral constant.
\end{definition}

\bekezdes\labelpar{ex:trap}
Assuming that the link of the singularity defined by $f$ is a rational homology
sphere (hence (\ref{eq:rhs}) is valid), all faces of $\Gamma_+(f)$
are either standard triangles or standard
trapezoids, cf \cite[2.3]{BN} or \ref{ss:NN} here.
(There is only one exception to this, namely the Newton polygon of
$z_1^{2a}+z_2^{2b}+z_3^{2c}$, $a,b,c$ pairwise relative primes, in which case
three edges contain
interior lattice points. Since this is a weighted homogeneous Newton diagram,
which case is covered by
part (a) of Theorem \ref{th:main} whose proof already was established,
we can assume that our Newton
diagram is not of this type. Alternatively, the interested reader might run the argument below
slightly modified to check this case too.)
Moreover, that edge of such a face which contains integral interior points
is necessarily contained in a coordinate plane.

Accordingly, in our  application and proof,  the  standard
polygons are $\bigtriangleup_{st}=F^{nb}_{v(i)}(Z_K-E)$, which
by  Lemma~\ref{lem:ZKE}, are the faces of $\Gamma_+(f)$ shifted by $-(1,1,1)$.

\begin{lemma}\labelpar{lem:comp} Consider $\bigtriangleup_{st}=F^{nb}_{v(i)}(Z_K-E)$ with
 long edge $I_1$, and an adjacent short edge $I_2$.
Then the map $(\frl_{st,1},\frl_{st,2}):A_{st}\cap\Z^3 \to \Z^2$ is a $\Z$-affine isomorphism.
\end{lemma}
\begin{proof}
Since $\bigtriangleup_{st}$
contains no integral points in its interior, neither does the
parallelogram spanned by the two primitive vectors supported on its edges
$I_1$ and $I_2$.  
\end{proof}

\bekezdes A polygon $\bigtriangleup$
in the affine plane $A$ is a {\it small scalar translate} of a standard
polygon  $\bigtriangleup_{st}$ if
\begin{equation}\label{eq:small1}
\mbox{there is a vector $V\in\R^3$ and $r\in(0,1)$ such that
$\bigtriangleup= V+ r\cdot \bigtriangleup_{st}$.}\end{equation}

Let $\bigtriangleup$ be a small scalar translate of $\bigtriangleup_{st}$, we denote
its edges by the same symbols $\{I_w\}_w$ ($I_1$ corresponding to the long edge).
Then for each $w$ there exists a unique integral
primitive affine function $\frl'_w$ on $A$ so that $\frl'_w|_{I_w}$ is constant with value in
$(-1,0]$,  and $\frl'_{w}|_{\bigtriangleup}\geq \frl'_{w}|_{I_w}$.
Additionally, we select $\epsilon_w\in\{0,1\}$ (for motivation see \ref{again}),
so that if a {\it short edge} contains a lattice point then $\epsilon_w$ might be 0 or 1, otherwise it is zero.
Finally, set $\frl_{w}:=\frl'_{w}-\epsilon_w$.

In our applications, $F_i^{cn}$ is a small scalar translate, whose equations
$\frl'_{w}(p)\geq \frl'_{w}|_{I_w}$  are replaced by
$\frl'_w(p)\geq \frl'_{w}|_{I_w}+\epsilon_w$ in order to obtain $F_i^{cn-}$.

\begin{lemma}\label{LAT} With the above notations one has:

(a) All the lattice points of $\bigtriangleup$ are situated on a line parallel to $I_1$.

(b) The function $t(\bigtriangleup_{st})\frl_1 + \sum_{w>1} \frl_w$ is constant (say $c$) on $A$.
Moreover,  $\max\{0,c + 1\}=|(\bigtriangleup\setminus \cup_{\epsilon_w=1}I_w)\cap \Z^3|$.
\end{lemma}
\begin{proof}
The proof splits into the cases when $\bigtriangleup_{st}$ is a triangle or it is a
trapezoid. First we fix the equations $\frl_{st,w}$ of the edges of  $\bigtriangleup_{st}$.
We abbreviate $t:=t(\bigtriangleup_{st})$.
 By Lemma~\ref{lem:comp}, in the case of a triangle,
we can take integral linear coordinates $(x,y)$ on $A_{st}$ so that $\frl_{st,1}=y$,
$\frl_{st,2}=x$ and $\frl_{st,3}=-x-ty$. In the case of trapezoid (see Lemma \ref{lem:13} too)
we can take
$\frl_{st,1}=y$,
$\frl_{st,2}=x$, $\frl_{st,3} =1-y$, and $\frl_{st,4}=-x-(t-1)y+t$.
In both cases the sum $t\frl_{st,1} + \sum_{w>1} \frl_{st,w}$  is constant on $A_{st}$.
Then note that the corresponding equations of $\bigtriangleup$, or their modifications by $\{\epsilon_w\}_w$,
are obtained from the above ones by adding constants, hence
$t\frl_{1} + \sum_{w>1} \frl_{w}$ is constant too. Evaluating on a lattice point
we get that this constant is an integer $c$.

In fact, in $A$ too,  we can take affine coordinate such that $\frl_1=y$ and $\frl_2=x$.

In $\bigtriangleup_{st}$ any line parallel to $I_2$  contains at most two lattice points
(two is realized by the endpoints of that edge). Therefore,  a parallel line to $I_2$
in $\bigtriangleup$ contains  at most one lattice point (by (\ref{eq:small1})).
Hence all the lattice points of $\bigtriangleup$ are on the line $y=0$. This shows (a).

In order to prove the second part of (b), first let us take $\epsilon_w=0$ for all $w$.

In the case of triangle $\frl_3=-x-ty+c$. If $c\geq 0$ then the lattice points of
$\bigtriangleup$ are exactly  $\{(0,0),\ldots, (c,0)\}$. If $c<0$ then $\bigtriangleup$
has no lattice point.

Next assume that $\bigtriangleup$ is a trapezoid. By the choice of $\frl_1=y$ and $\frl_2=x$, the vertex
$(a,b):=I_1\cap I_2$ is in $(-1,0]^2$. If $r\in (0,1)$ is the homothety factor of $\bigtriangleup$
then the other vertices are $(a,b+r)$, $(a+r,b+r)$ and $(a+tr,b)$. Since $a+r\in(-1, 1)$, we have two cases.

If $a+r\geq 0$, then   $\frl_3=-y$. Therefore, $\frl_4=-x-(t-1)y+c$ for some $c$, which is the sum
$t\frl_1+\sum_{w>1}\frl_w$.
Again, if $c\geq 0$ then the lattice points of
$\bigtriangleup$ are exactly  $\{(0,0),\ldots, (c,0)\}$, while  if $c<0$ then $\bigtriangleup$
has no lattice point.

If $a+r<0$, then $\frl_3=-y-1$. Clearly $\bigtriangleup\cap \Z^3=\emptyset$,
and by a computation $c<0$ too.

Finally notice that each $\epsilon_w=1$ decreases both sides of (b) by 1 till either we use all these
edges with $\epsilon_w=1$, or  $c$ becomes negative (hence both sides of (b) become zero).
\end{proof}

\begin{corollary}\labelpar{cor:I} {\bf (The proof of (\ref{eq:egy}))}
 Consider the step $\bar{z}_i\mapsto \bar{z}_i+E_{v(i)}$
where $v(i)$ is a face vertex  and $(E_{v(i)},\bar{z}_i)\leq 0$, as in \ref{bek:faces}. Then
(\ref{eq:egy}) holds:
\begin{equation*}\label{eq:pf:points_and_intersection}
 1-(E_{v(i)},\bar{z}_i)\leq |P_i|.
\end{equation*}
\end{corollary}

\begin{proof} Write $m_u=m_u(\bar{z}_i)$.
Recall that $\bigtriangleup=F_i^{cn}$ is given by equations $\ell_w(p)\geq r_{i,w}\cdot m_w$,
$F_i^{cn-}$ is given by $\ell_w(p)\geq r_{i,w}\cdot m_w+\epsilon_w$ ($w\in \calv_{v(i)}$).
Hence $\frl_w=\ell_w-\lceil r_{i,w}\cdot m_w\rceil-\epsilon_w$, and
$m_w\geq \lceil r_{i,w}\cdot m_w\rceil +\epsilon_w$ for all $w$.
(For the definition of $r_{i,w}$ see (\ref{eq:cone}).)

Th
en  $\calv_{v(i)}^e=\calv_{v(i)}$ (cf. Proposition~\ref{prop:OKA}(c)) and (\ref{eq:5}) imply
for any $p\in H^=_{v(i)}(\bar{z}_i)$
\begin{equation*} \label{eq:proof_of_main_prop}\begin{split}
 - (E_{v(i)},\bar{z}_i) &=- b_{v(i)} m_{v(i)} - \sum_{w\in\calv_{v(i)}} m_{w}
            =-b_{v(i)} \ell_{v(i)}(p) - \sum_{w\in\calv_{v(i)}} m_{w} \\
           & = \sum_{w\in\calv_{v(i)}}( \ell_w(p)-m_{w})\leq
           \sum_{w\in\calv_{v(i)}} \frl_w(p).
\end{split}
\end{equation*}
Since by assumption $(E_{v(i)},\bar{z}_i)\leq 0$, the sum $c:=  \sum_{w\in\calv_{v(i)}} \frl_w$
is non--negative. Hence, by Lemma \ref{LAT}, $c+1$ is the number of lattice points in $F_i^{cn-}$,
which is $|P_i|$ by Lemma~\ref{lem:cone}.
\end{proof}

\subsection{The proof of (\ref{eq:ketto})}\labelpar{ss:ketto}
The proof is based on the comparison of the divisorial  valuation/filtration  with the Newton
filtration.  For notations see subsection \ref{ss:weight}.

\begin{lemma} \labelpar{lem:EGZ} \cite{EGZ}
For any  $0\neq g\in\C\{z\}$ and  $v\in\calv$ we denote by $g_v$ the principal part of
$g$ with respect to $\wt_v$, the weight corresponding to $v$. Then for any non-zero $h\in\C\{z\}$,
$\wt_v(h) < {\rm div}_v(h)$ if and only if $h_v$ is divisible by $f_v$ over the
ring $\C\{z_1,z_2,z_3\}[z_1^{-1},z_2^{-1},z_3^{-1}]$.
\end{lemma}

As an immediate consequence we have

\begin{prop} \labelpar{prop:codim}
At each step $\bar{z}_i\mapsto \bar{z}_i+E_{v(i)}$  we have
\[
  |P_i|\leq \dim\ \frac{H^0(\X,\calO(-\bar{z}_i))}{H^0(\X,\calO(-\bar{z}_{i}-E_{v(i)}))}.
\]
\end{prop}
\begin{proof}
Consider the functions $z^p$ for $p\in P_i$. Since ${\rm wt}(z^p)={\rm div}(z^p)$, cf. \ref{rem:wtdiv},
we have $z^p\in H^0(\X,\calO(-\bar{z}_i)) \setminus H^0(\X,\calO(-\bar{z}_{i}-E_{v(i)}))$.
Therefore, it is enough to show that the family
$(z^p)_{p\in P_i}$ is linearly independent modulo
$H^0(\X,\calO(-\bar{z}_{i}-E_{v(i)}))$. This can be verifies as follows.
 Assume that there exist $a_p\in\C$ (at least one of them non--zero)
such that
\begin{equation} \label{eq:lin_indep}
 g:= \sum_{p\in P_i} a_p z^p \in H^0(\X,\calO(-\bar{z}_{i}-E_{v(i)})).
\end{equation}
 Then $g$ is its own principal part with
respect to the weight corresponding to $v(i)$. We have
$\wt_{v(i)}(g) = m_{v(i)}(\bar{z}_i)$, but by (\ref{eq:lin_indep})
${\rm div}_{v(i)}(g) \geq  m_{v(i)}(\bar{z}_{i}+E_{v(i)})=m_{v(i)}(\bar{z}_i)+1$.
By Lemma \ref{lem:EGZ} this means that for some
$h \in\C\{z_1,z_2,z_3\}[z_1^{-1},z_2^{-1},z_3^{-1}]$ one has $g = f_{v(i)} \cdot h$. But
this is impossible, since the support of $g$ lies on a segment (cf. Lemma \ref{LAT}(b))
 whereas the same can not be said about $f_{v(i)}$ (since $v(i)$ is a face vertex),
  hence about  $f_{v(i)} \cdot h$ neither.
\end{proof}

\end{document}